\documentclass[a4paper,12pt]{amsart}
\usepackage[utf8]{inputenc}
\usepackage{amsthm}
\usepackage{amsmath}
\usepackage{bm}
\usepackage{enumitem}
\usepackage{amsfonts}
\usepackage{amssymb}
\usepackage{mathtools}
\usepackage{appendix}
\usepackage{mathrsfs}
\usepackage{setspace}
\usepackage{comment}
\usepackage{xcolor}
\usepackage{todonotes}
\usepackage{thmtools} 
\usepackage[foot]{amsaddr}
\RequirePackage[a4paper,top=2.54cm,bottom=2.54cm,left=1.90cm,right=1.90cm,%
                headsep=1em,includehead,includefoot]{geometry}
\usepackage[pdfdisplaydoctitle,colorlinks,breaklinks,urlcolor=blue,linkcolor=blue,citecolor=blue]{hyperref} 
\usepackage[nameinlink,capitalise]{cleveref}

\newcommand{\A}{\mathcal{A}}

\newenvironment{acknowledgements}{%
  \begin{abstract}
}{%
  \end{abstract}
}

\newcommand{\EE}{\mathbb{E}}

\renewcommand{\H}{\mathcal{H}}

\renewcommand{\L}{\mathscr{L}}

\newcommand{\N}{\mathbb{N}}

\newcommand{\PP}{\mathbb{P}}

\newcommand{\Q}{\mathbb{Q}}
\newcommand{\R}{\mathbb{R}}

\newcommand{\T}{\mathbb{T}}

\newcommand{\Z}{\mathbb{Z}}

\let\div\relax

\DeclareMathOperator{\div}{div}

\renewcommand{\epsilon}{\varepsilon}
\renewcommand{\setminus}{\smallsetminus}
\newcommand{\eps}{\epsilon}

\newcommand{\one}{\bm{1}}

\newcommand{\pa}[1]{\left(#1\right)}

\newcommand{\norm}[1]{\left\|#1\right\|}

\newcommand{\sumkj}{\sum_{\substack{k\in \Z^{3}_0\\ j\in \{1,2\}}}}

\newcommand{\skj}{\sigma_{k,j}^{n}}

\newcommand{\smkj}{\sigma_{-k,j}^{n}}

\newcommand{\tkj}{\theta_{k,j}^{n}}

\newcommand{\expt}[2][]{\mathbb{E}_{#1}\left[#2\right]}
\newtheorem{theorem}{Theorem}[section]
\newtheorem{definition}[theorem]{Definition}

\newtheorem{corollary}[theorem]{Corollary}
\newtheorem{lemma}[theorem]{Lemma}
\newtheorem{proposition}[theorem]{Proposition}

\theoremstyle{remark}
\newtheorem{remark}[theorem]{Remark}

\numberwithin{equation}{section}

\title[Vlasov equations and Young measures for passive scalar and vector
advection]{Background Vlasov equations and Young measures for passive scalar and vector
advection equations under special stochastic scaling limits}
\author[F. Butori]{Federico Butori}
\address{Scuola Normale Superiore, Piazza dei Cavalieri, 7, 56126 Pisa, Italia}
\email{\href{mailto:federico.butori at sns.it}{federico.butori at sns.it}}
\author[F. Flandoli]{Franco Flandoli}
\address{Scuola Normale Superiore, Piazza dei Cavalieri, 7, 56126 Pisa, Italia}
\email{\href{mailto:franco.flandoli at sns.it}{franco.flandoli at sns.it}}
\author[E. Luongo]{Eliseo Luongo}
\address{Scuola Normale Superiore, Piazza dei Cavalieri, 7, 56126 Pisa, Italia}
\email{\href{mailto:eliseo.luongo at sns.it}{eliseo.luongo at sns.it}}
\author[Y. Tahraoui]{Yassine Tahraoui}
\address{Scuola Normale Superiore, Piazza dei Cavalieri, 7, 56126 Pisa, Italia}
\email{\href{mailto:yassine.tahraoui at sns.it}{yassine.tahraoui at sns.it}}
\date\today
\date\today
\allowdisplaybreaks

\begin{document}
\begin{abstract}
In the last few years it was proved that scalar passive quantities subject to suitable stochastic transport noise, and more recently that also vector passive quantities subject to suitable stochastic transport and stretching noise, weakly converge to the solutions of deterministic equations with a diffusion term. In the background of these stochastic models, we introduce stochastic Vlasov equations which gives additional information on the fluctuations and oscillations of solutions: we prove convergence to non-trivial Young measures satisfying limit PDEs with suitable diffusion terms. In the case of a passive vector field the background Vlasov equation adds completely new statistical information to the stochastic advection equation.
\end{abstract}
\maketitle
\section{Introduction}

    This work introduces a background stochastic Vlasov equation behind stochastic transport and advection equations which gives additional information on the fluctuations and oscillations of solutions. It is based on Young measures; although the use of Young measures in PDEs is classical, the introduction of these particular Vlasov equations seems to be new. 
    We develop first the theory for the stochastic transport of a passive scalar, where the result is new but just enriches incrementally the existing understanding. More importantly, then we develop the theory for a passive vector field, where also stretching acts besides transport. Here the results based on the background Vlasov equation add completely new statistical information to the stochastic advection equation. 

\subsection{Stochastic passive scalar}\label{sec:intro_passive_scalar}
Consider the stochastic transport equation on the $d$-dimensional torus
$\mathbb{T}^{d}$%
\begin{align*}
d\theta^n\left( t, x\right)  &=\sum_{k,j}\skj\left(  x\right)
\cdot\nabla\theta^n\left(  t,x\right)  \circ dW_{t}^{k,j}  \\
\theta^n|_{t=0}  &  =\overline{\theta}_{0}.
\end{align*}
Under suitable assumptions described below we know, since \cite{galeati2020convergence}, that
$\theta^n$ weakly converges to $\overline{\theta}$, solution of%
\begin{align*}
\partial_{t}\overline{\theta}\left(  t,x\right)   &  =\kappa_{T}%
\Delta\overline{\theta}\left( t, x\right) \\
\overline{\theta}|_{t=0}  &  =\overline{\theta}_{0}%
\end{align*}
for a suitable constant $\kappa_{T}>0$. It is clear that strong convergence in $L^2$ cannot hold, since the $L^2$-norm is preserved by the stochastic transport equation and not by the limit heat equation (different is the case when a dissipation is already present in the transport equation, see the recent work \cite{agresti}). The weak convergence is only weak; here we reinforce this fact by showing that the limit Young measure is not trivial. Indeed we prove that, for
every test functions $\varphi\in C_{b}\left(  \mathbb{R}\right)  $, $\psi\in
C_{b}\left(  \mathbb{T}^{d}\right)  $, for every $t\in\left[  0,T\right]  $ we
have%
\[
\lim_{n\rightarrow\infty}\int_{\T^d}\varphi\left(  \theta^n\left(  t,x\right)
\right)  \psi\left(  x\right)  dx=\int_{\T^d}\left(  \int_{\R}\varphi\left(
\theta\right)  \mu\left(  t,x,d\theta\right)  \right)  \psi\left(  x\right)
dx
\]
for a suitable probability measure $\mu\left(  t,x,d\theta\right)  $ on
$\mathbb{R}$, measurably dependent on $\left( t, x\right)  $, having the
property%
\[
\int_{\R}\theta\mu\left(  t,x,d\theta\right)  =\overline{\theta}\left(  t,x\right)
.
\]
The probability measure $\mu\left(  t,x,d\theta\right)  $ is the Young measure
associated to the weak convergence of $\theta^n$ to $\overline{\theta}$, and
solves, in the distributional sense, the equation%
\begin{align*}
\partial_{t}\mu\left(  t,x,d\theta\right)   &  =\kappa_{T}\Delta_{x}\mu\left(
t,x,d\theta\right) \\
\mu\left(  0,x,d\theta\right)   &  =\delta_{\overline{\theta}_0\left(  x\right)
}\left(  d\theta\right)  .
\end{align*}
As shown by the example of \autoref{rem2} below, this Young measure is not trivial, as it should be when strong convergence does not hold.

Locally around each point $x_0$ (choose, in the limit above, a probability density function $\psi$ concentrated around $x_0$) the values of $\theta^n(x,t)$ have large fluctuations/oscillations: the complex Lagrangian dynamics produced by the transport noise mixes the different values of the initial condition $\overline{\theta}_0(x)$ to such a high degree that everywhere, locally, there are high variations; before taking the limit as $n\rightarrow\infty$. Then, in the limit $n\rightarrow\infty$, these local oscillations around a point $x_0$ are summarized, at $x_0$, by the non-trivial measure $\mu\left(  t,x_0,d\theta\right)$. This is just the general interpretation of Young measures but it may be convenient to recall it for the understanding of the result, in connection with the mixing properties of the flow.

The result above comes from the following fact: in the background of the stochastic transport equation above there is a
stochastic Vlasov equation for measures%
\begin{align*}
d\mu_{n}\left(  t,x,d\theta\right)  &=\sum_{k,j}\operatorname{div}_{x}\left(
\skj\left(  x\right)  \mu_{n}\left(  t,x,d\theta\right)  \right)
\circ dW_{t}^{k,j} \\
\mu_{n}\left(  0,x,d\theta\right)   &  =\delta_{\overline{\theta}_0\left(  x\right)
}\left(  d\theta\right)
\end{align*}
and $\mu\left( t ,x,d\theta\right)  $ is the weak limit of $\mu_{n}\left(
t,x,d\theta\right)  $. It is the Vlasov equation of the dynamics%
\begin{align*}
dX_{t}^{x,n}  &  +\sum_{k,j}\sigma_{k}^{n}\left(  X_{t}^{x,n}\right)  \circ
dW_{t}^{k}=0\\
\frac{d\theta_{t}^{x,n}}{dt}  &  =0
\end{align*}
which represents the Lagrangian view of the transport equation, where%
\[
\theta_{t}^{x,n}=\theta^n\left(t,  X_{t}^{x,n}\right)  .
\]
Let us remark that in homogenization theory there are several results about the diffusion limit, see for instance \cite{MajdaCramer}. However they always require a dissipation in the original transport equation, hence it is not clear that one can prove for such models a result like the one above for the purely-transport stochastic problem.

\subsection{Stochastic passive vector field}\label{magnetic_field_intro}

The result described so far is just a minor reformulation of the fundamental
result of \cite{galeati2020convergence}, just emphasizing the local oscillation properties of $\theta^n(t,x)$. The real progress comes in the case of vector advected quantities.
Consider the passive vector equation
\begin{align*}
dB^{n}\left( t,x\right)  &=\sum_{k,j}\left(  \skj\left(  x\right)
\cdot\nabla B^{n}\left( t, x\right)  -B^{n}\left(  t,x\right)  \cdot
\nabla\skj\left(  x\right)  \right)  \circ dW_{t}^{k,j} \\
B^{n}|_{t=0} &  =\overline{B}_{0}.
\end{align*}
Also here, very recently \cite{butori2024mean}, it has been proved that, in a
suitable weak sense, under suitable assumptions on $\left(  \skj\right)  $, $B^{n}\left(t,  x\right)  $ weakly converges to $\overline
{B}\left( t, x\right)  $ solution of
\begin{align*}
\partial_{t}\overline{B} &  =\mathcal{D}\overline{B}\\
\overline{B}|_{t=0} &  =\overline{B}_{0}%
\end{align*}
where $\mathcal{D}$ is a linear operator, described below, in some cases just
identically equal to zero. The Lagrangian dynamics behind the vector advection
equation is much richer:%
\begin{align*}
dX_{t}^{x,n} &  +\sum_{k,j}\skj\left(  X_{t}^{x,N}\right)  \circ
dW_{t}^{k,j}=0\\
{db_{t}^{x,n}} &  +\sum_{k,j}b_{t}^{x,n}\cdot\nabla\skj\left(  X_{t}^{x,n}\right)
\circ dW_{t}^{k,j}=0%
\end{align*}
where%
\[
b_{t}^{x,n}=B^{n}\left(t,  X_{t}^{x,n}\right)  .
\]
The associated Vlasov equation is now (we denote $\mu_{n}\left(
t,x,db\right)  $ simply by $\mu_{n}$)
\begin{align*}
d\mu_{n}&=\sum_{k,j}\operatorname{div}_{x}\left(  \skj\left(  x\right)
\mu_{n}\right)  \circ dW_{t}^{k,j}+\sum_{k,j}\operatorname{div}_{b}\left(
b\cdot\nabla \skj\left(  x\right)  \mu_{n}\right)  \circ dW_{t}^{k,j} \\
\mu_{n}\left(  0,x,db\right)   &  =\delta_{\overline{B}_0\left(  x\right)  }\left(
db\right)  .
\end{align*}
Under suitable assumptions on $\left(  \skj\right)  $ we prove that
$\mu_{n}\left( t, x,db\right)  $ weakly converges to $\mu\left(
t,x,db\right)  $, solution of
\begin{align*}
\partial_{t}\mu &  =\mathcal{L}\mu\\
\mu\left(  0,x,db\right)   &  =\delta_{\overline{B}_0\left(  x\right)  }\left(
db\right)
\end{align*}
for a suitable differential operator $\mathcal{L}$. Opposite to the case of
the passive scalar $\theta$, here $\mathcal{L}$ may include second order terms
in $b$ which produce a regularization in the $b$-variable of the initial
condition $\delta_{\overline{B}_{0}\left(  x\right)  }\left(  db\right)  $. Again we have
the compatibility relations (now for $\varphi\in C_{b}\left(  \mathbb{R}%
^{3}\right)  $, $\psi\in C_{b}\left(  \mathbb{T}^{3}\right)  $)%
\begin{equation}
\lim_{n\rightarrow\infty}\int_{\T^3}\varphi\left(  B^{n}\left(  t,x\right)  \right)
\psi\left(  x\right)  dx=\int_{\T^3}\left(  \int_{\R^3}\varphi\left(  b\right)  \mu\left(
t,x,db\right)  \right)  \psi\left(  x\right)  dx\label{1}%
\end{equation}
identifying $\mu\left(  t,x,db\right)  $ as the Young measure associated to
the sequence $\left(  B^{n}\right)  $ and
\begin{equation}
\int_{\R^3} b\mu\left(  t,x,db\right)  =\overline{B}\left(  t,x\right)  .\label{2}%
\end{equation}

Up to here, the interpretation of the result would be similar to the scalar case, namely that the reformulation with Young measures gives more insight about the oscillations. However, here the oscillations are not only produced by the mixing mechanism which moves different values of the initial condition close each other. Here the stretching  may enhance the oscillations. One could believe that the equation for  $\overline{B}$ is sufficient to capture the effect of stretching, but it is not so. On the contrary, the background Vlasov equation introduced here does. Let us show this phenomenon.

Consider the case (inspired by \cite{flandoli2024}, devoted to the dynamics of polymers in turbulent fluids, which is different from many viewpoints but shares some elements) when suitable assumptions on $\left(
\sigma_{k}^{n}\right)  $ imply
\[
\mathcal{D}\overline{B}=0,
\]
namely%
\begin{equation}
\overline{B}\left(  t,x\right)  =\overline{B}_{0}\left(  x\right)  .\label{3}%
\end{equation}
The equation for $\mu\left(  t,x,db\right)  $ is%
\begin{align*}
\partial_{t}\mu &  =\operatorname{div}_{b}\left(  \mathcal{L}\left(  b\right)
\nabla_{b}\mu\right)  \\
\mu\left( 0, x,db\right)   &  =\delta_{\overline{B}_0\left(  x\right)  }\left(
db\right)
\end{align*}
where the matrix-valued function $\mathcal{L}\left(  b\right)  $ is described below.
Except for $b=0$, this matrix is non-singular and provides solutions with full support in 
$b$; see \autoref{Corollary B6 on support} and its consequence by \autoref{PropLargeValues}, namely that for every $(x_0,t)$ and $R,r>0$, for large enough $n$,
\[
\mathcal{L}_{3}\left\{  x\in B\left(  x_{0},r\right)  :\left\vert B^{n}\left(
t,x\right)  \right\vert \geq R\right\} >0.
\]
Notice also that the stationary equation%
\[
\operatorname{div}_{b}\left(  \mathcal{L}\left(  b\right)  \nabla_{b}f\left(  b\right)
\right)  =0
\]
corresponding to uniform conditions in $x$, is solved by an explicit rotation-invariant
distribution, a probability density function $f(b)$. This density has full support on positive values. 

Notice that in the scaling limit regime when $\mathcal{D}\overline{B}=0$ holds, transport becomes irrelevant, hence the non-triviality of the Young measure is not due to transport mixing but only to turbulent stretching.

The interpretation of these results then is that the
stochastic vector advection equation, for large $n$, due to the turbulent Lagrangian stretching produced by $\left(  \skj\right)  $, develops wide
variations (and in particular large values) of $B^{n}\left(  t,x\right)  $ at small scale, so that
in an infinitesimal volume around any point $x_0$ there is a ``density" of values
of $B^{n}\left( t, x_0\right)  $ (just a sort of empirical density, for finite
$n$) converging in the limit to a measure $\mu\left(  t,x_0,db\right)  $ with a
true density. For large times the distribution in $x$ is uniform and this
``density" of $B$-values is described by $f\left(  b\right)  $. In spite of the
fact that in the average we have \eqref{3}. Therefore the background Vlasov
equation provides new information, not visible at the level of the vector
advection equation. The correspondence of $f\left(  b\right)  $ with
statistical properties observed for magnetic fields in turbulent plasma and
other examples will be discussed elsewhere.
\subsubsection*{Plan of the Paper}
In \autoref{sec:prelim_main_result} we provide the functional analytic setup in order to describe rigorously our main results: \autoref{theorem_passive_scalar} and \autoref{main_thm_magnetic_field}. The proof of \autoref{theorem_passive_scalar} is the content of \autoref{sec:passive_scalar}, while the proof of \autoref{main_thm_magnetic_field} is the content of \autoref{sec:magn_field}. We conclude the paper with two appendices.  \autoref{appendix PDE magnetic Field} is devoted to the study of the well-posedness of the stochastic equation for the ideal magnetic field, while we add some comments on the interpretation of our results in \autoref{appendix:rmk_physics}.
\section{Functional Analytic Setup and Main Results}\label{sec:prelim_main_result}
\subsection{Notation}\label{Preliminaries}
Let us start setting some classical notation and properties of operators in the periodic setting before describing the main contributions of this work. We refer to monographs \cite{Engel_Nagel,temam1995navier,temam2001navier, trie1995fun} for a complete discussion on the topics shortly recalled in this section.\\
In the following we denote by $\mathbb{T}^d=\mathbb{R}^d/(2\pi\mathbb{Z})^d$ the $d$ dimensional torus and by $\Z^d_0=\Z^d\setminus \{0\}$ the $d$ dimensional lattice made by points with integer coordinates.
On $\Z^d_0$ we introduce the partition $\Z^d_0=\Gamma_{d,+}\cup \Gamma_{d,-}$ such that $\Gamma_{d,+}=-\Gamma_{d,-}.$\\
Given a function $g:~\T^d~\rightarrow ~\R\cup ~\{+\infty\}$ we denote by $\int_{\T^d}^* g(x) dx$ the outer integral of $g$, i.e. 
\begin{align*}
    \int_{\T^d}^* g(x) dx=\inf_{\substack{G\geq g\\
    G\ \textit{Borel measurable}}} \int_{\T^d} G(x) dx.
\end{align*} Points of $\T^d\times \R^k$ will be denoted by bold letters. Let $a,b$ be two positive numbers, we write $a\lesssim b$ if there exists a positive constant $C$ such that $a \leq C b$ and $a\lesssim_\xi b$ if we want to highlight the dependence of the constant $C$ on a parameter $\xi$. \\ 
 Let $ \left({H}^{s,p}(\T^d),\norm{\cdot}_{{H}^{s,p}}\right),\ s\in\mathbb{R},\ p\in (1,+\infty)$ be the Bessel spaces of periodic functions with zero mean. In case of $p=2$, we simply write ${H}^{s}(\T^d)$ in place of ${H}^{s,2}(\T^d)$ and we denote by $\langle \cdot,\cdot\rangle_{{H}^s}$ the corresponding scalar products.
In case of $s=0$, we write ${L}^p(\T^d)$ instead of ${H}^{0,p}(\T^d)$ and in case of $p=2$ we neglect the subscript in the notation for the norm and the inner product. With some abuse of notation, for $s>0$, we denote the duality between ${H}^{-s}$ and ${H}^s$ by $\langle\cdot,\cdot\rangle$.\\
A different characterization of the fractional Sobolev spaces above can be given in terms of powers of the Laplacian. Denoting by \begin{align*}
    \Delta: \mathsf{D}(\Delta)\subseteq {L^2}(\T^d)\rightarrow {L}^2(\T^d),
\end{align*}
where $\mathsf{D}(\Delta)={H}^2(\T^d)$, it is well known that ${\Delta}$ is the infinitesimal generator of analytic semigroup of negative type that we denote by $e^{t\Delta}: L^2(\T^d)\rightarrow  L^2(\T^d)$ and moreover for each $s>0$, $\mathsf{D}((-\Delta)^{s})$ (resp. $\left(\mathsf{D}((-\Delta)^{s})\right)^{\prime}$) can be identified with $ H^{2s}(\T^d)$ (resp. $ H^{-2s}(\T^d)$). The semigroup $e^{t\Delta}$ is associated to an integral kernel $p(t,x)$ such that \begin{align*}
    e^{t\Delta}u_0(x)=\int_{\T^d}p(t,x-y)u_0(y)dy.
\end{align*}\\
Similarly, we introduce the Bessel spaces of vector fields
\begin{align*}
    {H}^{s,p}(\T^d;\mathbb{R}^d)&= \{u=(u^1,\dots, u^d)^t:\ u^1,\dots, u^d\in {H}^{s,p}(\T^d)\},\\  \langle u, v\rangle_{{{H}^s}}&=\sum_{j=1}^d\langle u^j,v^j\rangle_{{H}^s}, \quad \text{for } s\in \mathbb{R}.
\end{align*}
Again, in case of $s=0$ we write ${L}^p(\T^d;\R^d)$ instead of ${H}^{0,p}(\T^d;\R^d)$ and we neglect the subscript in the notation for the norm and the scalar product in case of $p=2$. 
Lastly we denote by $\mathbf{H}^{s}$ the closed subspace of ${H}^s(\T^d;\R^d)$ made by zero mean, divergence free vector fields with norm induces by ${H}^s(\T^d;\R^d)$ and in order to ease the notation we, simply, write
\begin{align*}
    H=\mathbf{H}^0,\ V=\mathbf{H}^1.
\end{align*}
We conclude this subsection introducing some classical notation when dealing with stochastic processes taking values in separable Hilbert spaces. Let $Z$ be a separable Hilbert space, with associated norm $\| \cdot\|_{Z}$. We denote by $C_{\mathcal{F}}\left(  \left[  0,T\right]  ;Z\right) $ (resp. $C_{\mathcal{F}}\left(  \left[  0,T\right]  ;Z_w\right) $) the space of continuous (resp. weakly continuous) adapted processes $\left(  X_{t}\right)  _{t\in\left[
0,T\right]  }$ with values in $Z$ such that
\[
\mathbb{E} \bigg[ \sup_{t\in\left[  0,T\right]  }\left\Vert X_{t}\right\Vert
_{Z}^{2}\bigg]  <\infty
\]
and by $L_{\mathcal{F}}^{p}\left(  0,T;Z\right),\ p\in [1,\infty),$ the space of progressively
measurable processes $\left(  X_{t}\right)  _{t\in\left[  0,T\right]  }$ with
values in $Z$ such that
\[
    \mathbb{E} \bigg[ \int_{0}^{T}\left\Vert X_{t}\right\Vert _{Z}^{p}dt \bigg]
<\infty.
\]

\subsection*{Some Preliminaries on Young Measures}
Let us denote by $\mathcal{P}$ the set of probability measures on $\T^d\times \R^k$. We recall here some standard facts on probability measures on $\T^d\times \R^k$ for which we refer to \cite[Section 1.1]{GiPa20}, \cite[Section 5.1]{ambrosio2005gradient}. 
Let us start denoting by $\{\mathbf{x}^n\}_{n\in \N}$ a dense set of $\T^d\times \R^k$ and
\begin{align*}
    \mathcal{A}&=\{(a-b\lvert \mathbf{x}-\mathbf{x}^n\rvert)\vee c,\quad a,b,c\in \Q,\ b\geq 0,\ n\in\N\},\\ \mathbf{A}&=\{f_1\vee f_2\vee \dots \vee f_n,\quad f_i\in \mathcal{A},\ n\in\N\}\\
    \mathcal{G}&=\{g\in \mathbf{A}\cup(-\mathbf{A}): \norm{g}_{C_b(\T^d\times \R^k)}\leq 1\}.
\end{align*}
The set $\mathcal{G}$ is countable, therefore introducing an enumeration of its elements $\{ g_i\}_{i\in \N}$ and the following distance on the probability measures on $\T^d\times \R^k$
\begin{align*}
    \rho(\mu,\nu)=\sum_{i\in \N}\frac{\left\lvert \int_{\T^d\times \R^k} g_i(\mathbf{x}) (\mu-\nu)(d\mathbf{x}) \right\rvert}{2^{i}},
\end{align*}
it holds
\begin{theorem}\label{metric_measures}
The weak topology on $\mathcal{P}(\T^d\times \R^k)$ is induced by the distance $\rho$.
\end{theorem}
\begin{remark}
    Due to our definition, the functions $g_i\in \mathcal{G}$ are uniformly continuous on $\T^d\times \R^k$, since 
    \begin{align*}
        g_i=\pm(f_1\vee\dots\vee f_n )
    \end{align*}
    for some $n\in \N$ and $\{f_j\}_{j\in \{1,\dots,n\}}\subseteq \mathcal{A}$. In particular each $f_j$ is constant outside a certain ball $B_j\subseteq \T^d\times \R^k$. Since $n$ is finite, this implies that also $g$ is constant outside $B=\cup_{j=1}^n B_j$, from which the required uniform continuity of $g_i$ follows.
\end{remark}
Secondly let us denote by $\mathcal{Y}$ the subset of non negative Borel measures on $\T^d\times \R^k$ given by the Young measures on $\T^d\times \R^k$. We refer to \cite[Chapter 4.3]{attouch2014variational} and \cite{valdier1994course} for a complete discussion on the topic of Young measures. We recall that $\mu\in \mathcal{Y}$ if it non negative and for each $ A\subseteq \T^d$ Borel measurable, then
\begin{align*}
    \mu(A\times \R^k)=\mathcal{L}_d(A),
\end{align*}
where $\mathcal{L}_d$ is the Lebesgue's measure on $\T^d$. The latter, by disintegration theorem for measures, implies that there exists a family of probability measures on $\mathbb{R}^k$ parametrized by $x\in \T^d$, $\mu(x,db)$, such that for each $f\in C_b(\T^d\times \R^k)$
\begin{align*}
    \int_{\T^d\times \R^k} f(x,b) d\mu(x,b)=\int_{\T^d\times \R^k} f(x,b) \mu(x,db)dx
\end{align*}
and the map $x\rightarrow \mu(x,db)$ is weakly$^*$ measurable. 
Given $\mu_n,\ \mu\in \mathcal{Y}$ we say that $\mu_n$ converges to $\mu$ in the sense of Young measures (shortly $\mu_n\stackrel{\mathcal{Y}}{\rightarrow}\mu$), if $\forall f\in C_b(\T^d\times \R^k)$
\begin{align*}
    \int_{\T^d\times \R^k}f(x,b)\mu_n(x,db)dx\rightarrow \int_{\T^d\times \R^k}f(x,b)\mu(x,db)dx.
\end{align*}
Since the measure of the $d$ dimensional torus is finite, convergence in the sense of Young measures is characterized by \autoref{metric_measures}.
We recall now a compactness criteria for Young measures. 
\begin{theorem}\label{compactness_young_measures}
Let $\{\mu_n\}_{n\in \N}\subseteq \mathcal{Y}$ such that 
\begin{align*}
    \forall \delta>0,\ x\in \T^d\  \exists A_x\subseteq \R^k \textit{ compact: } \ \sup_{n\in \N}\int_{\T^d}^* \mu_n(x,db)(A_x^c) dx<\delta,
\end{align*}
then there exists a subsequence $n_k$ and a Young measure $\mu$ such that
\begin{align*}
    \mu_{n_k}\stackrel{\mathcal{Y}}{\rightarrow}\mu.
\end{align*}
    
\end{theorem}
Let us denote $B(x,R)\subset \T^d$ the ball of radius $R$ centered at $x$. Combining \autoref{metric_measures}, \autoref{compactness_young_measures} and Ascoli-Arzelà Theorem for functions with values in metric spaces \cite{kelley2017general} implies the following
\begin{lemma}\label{compact_sets}
Fix $\delta>0$, introduce for each $i,k\in \N$ some positive quantities $R=R(k,\delta)$ and $\theta=\theta(k,\delta,i)$, then the set 
\begin{align*}
    K^{\delta}:=\cap_{k\in \N}\{\mu\in C([0,T];\mathcal{Y}):\  &\operatorname{sup}_{t\in [0,T]}\int_{\T^d}^* \mu_t(x)(B(0,R))dx\leq \frac{1}{k},\\ &\sup_{i\in \N}\sup_{\lvert t-s\rvert< \theta}\left\lvert \int_{\T^d\times \R^k} g_i(\mathbf{x}) (\mu_t-\mu_s)(d\mathbf{x})\right\rvert \leq \frac{1}{k}\}
\end{align*}
is relatively compact in $C([0,T];\mathcal{Y})$.
\end{lemma}
\begin{proof}
Given a sequence $(\mu^n)\subset K^{\delta},$
by Ascoli-Arzelà Theorem, the claim is equivalent to the fact that
\begin{enumerate}
    \item $\forall t\in [0,T]$ $\{\mu^n_t\}$ is relatively compact in $\mathcal{Y}$.
    \item Uniform equicontinuity holds.
\end{enumerate}
The first condition is immediately true due to \autoref{compactness_young_measures}, concerning the second one, let us fix $\eps$ and find $\overline{k},\ \overline{i}\in \N$ such that
\begin{align*}
    \overline{k}>{\frac{2}{\eps}},\quad \sum_{i\geq \overline{i}}\frac{\lvert \T^d\rvert}{2^{i-1}}\leq \frac{\eps}{2}
\end{align*}
and $\gamma=\min\{\theta(\overline{k},\delta,i),\quad i\in\{1,\dots,\overline{i}\}\}>0$.
Therefore, for each $n\in \N,\ \lvert t-s\rvert\leq \gamma$
\begin{align*}
    \rho(\mu^n_t,\mu^n_s)&\leq \sum_{i=1}^{\overline{i}-1} \frac{\left\lvert \int_{\T^d\times \R^k} g_i(\mathbf{x}) (\mu^n_t-\mu^n_s)(d\mathbf{x})\right\rvert}{2^i}+2\lvert \T^d\rvert\sum_{i\geq \overline{i}}\frac{1}{2^{i}}\\ & \leq \frac{\eps}{2}\sum_{i=1}^{\overline{i}-1} \frac{1}{2^i}+\frac{\eps}{2}<\eps.
\end{align*}
\end{proof}
\subsection{Main Results}\label{subsec:Main results}
\subsection*{Stochastic Passive Scalar}
Let us start describing the coefficients of the stochastic transport term appearing in the transport equation of \autoref{sec:intro_passive_scalar}. For each $n\in \N$, $\{W_t^{k,j}\},$ ${{k\in \Z^{d}_0, \  j\in \{1,\dots, d-1\}}}$ we introduce a sequence of complex-valued Brownian motions adapted to $\mathcal{F}_t$ 
    such that
$W^{-k,j}_t=(W^{k,j}_t)^*$ and 
\begin{align*}
    \expt{W^{k,j}_1,{W^{l,m}_1}^*}&=\begin{cases}
        2 & \textit{if } k=l, \ m=j\\
        0 &  \textit{otherwise} ;
    \end{cases}
    \quad \left[W^{k,j}_\cdot, W^{l,m}_\cdot\right]_t=\begin{cases}
        2t & \textit{if } k=-l, \ m=j\\
        0 &  \textit{otherwise }.    
        \end{cases}
\end{align*}
Secondly, for each $ k\in \Z^{d}_0,\ j\in\{1,\dots, d-1\}$ we denote by $\skj(x)=\tkj  a_{k,j}e^{ik\cdot x}$, where $\{\frac{k}{\lvert k\rvert}, a_{k,1},\dots, a_{k,d-1}\}$ is an orthonormal system of $\R^d$ for $k\in \Gamma_{d,+}$ and $a_{k,j}=a_{-k,j}$ if $k\in \Gamma_{d,-}$, while \begin{align}\label{definition thetakj_passive_scalar}
    \tkj &=\sqrt{\frac{d \kappa_T}{(d-1)c_n}}\one_{\{n\leq \lvert k\rvert\leq 2n\} }\frac{1}{\lvert k\rvert^{d/2}},\quad c_n=\sum_{\substack{k\in \Z^d_0\\ n\leq \lvert k\rvert \leq 2n}}\frac{1}{\lvert k\rvert^{d}}=O(1),\quad \kappa_T>0 . 
\end{align}
With our choices, we ensure that $W^n$ takes values in a space of real valued functions.
It is well known that the family $\left\{a_{k,j}\frac{1}{\sqrt{(2\pi)^d}}e^{ik\cdot x}\right\}, \ {{k\in \Z^{d}_0, \ j\in\{1,\dots, d-1\}}}$ is a complete orthogonal systems of $H$ made by eigenfunctions of $-\Delta$. Thanks to this choice of the coefficients, the equation for the passive scalar of \autoref{sec:intro_passive_scalar} can be rewritten in It\^o form as 
\begin{align}\label{ito_PDE_passive_scalar}
 \begin{cases}
     d\theta^n_t&=\kappa_T\Delta \theta^n_t dt+\sum_{\substack{k\in \Z^d_0\\ j\in \{1,\dots,d-1\}}}\skj\cdot \nabla \theta^n_t dW^{k,j}_t,\\
     \theta^n_0&=\overline{\theta}_0,
 \end{cases}   
\end{align}
see \cite[Section 2.3]{galeati2020convergence} for details. We are interested to analytically weak solution of \eqref{ito_PDE_passive_scalar} as described by the following definition.
\begin{definition}\label{passive_scalar_weak}
A stochastic process  $\theta^n\in C_{\mathcal{F}}([0,T];{L}^2(\T^d))\cap L^2_{\mathcal{F}}(0,T;{H}^1(\T^d))$ is a weak solution of equation \eqref{ito_PDE_passive_scalar} if $\mathbb{P}-$a.s.
\begin{align*}
 \langle \theta^n_t,\phi\rangle-\langle \overline{\theta}_0,\phi\rangle&=-\kappa_T\int_0^t \langle \nabla \theta^n_s,\nabla \phi\rangle ds+\sum_{\substack{k\in \Z^d_0\\ j\in \{1,\dots,d-1\}}} \int_0^t \langle\skj \cdot\nabla \theta^n_s ,\phi\rangle dW^{k,j}_s   
\end{align*}
for each $\phi \in {H}^1(\T^d),\ t\in [0,T]$.    
\end{definition}
The well-posedness of \eqref{ito_PDE_passive_scalar} in the sense of \autoref{passive_scalar_weak} is guaranteed by the following standard result, which collect also some a priori estimates crucial for the following arguments.
\begin{theorem}\label{well_posedness_estimates_passive_scalar}
For each $\overline{\theta}_0\in {H}^1(\T^d)$ there exists a unique weak solution of \eqref{ito_PDE_passive_scalar} in the sense of \autoref{passive_scalar_weak}. Moreover $\theta^n\in C_{\mathcal{F}}(0,T;{H}^1_w(\T^d))$ and 
\begin{align*}
    d\norm{\theta^n_t}^2&=0.
\end{align*}
\end{theorem}
\autoref{well_posedness_estimates_passive_scalar} might be well known to the experts. Therefore, we omit its proof which however can be obtained following
verbatim the one provided below in a more complex framework for the magnetic field, see \autoref{well_posedness_estimates_MF} and \autoref{appendix PDE magnetic Field}. As discussed in \autoref{sec:intro_passive_scalar}, the results of \cite{galeati2020convergence, flandoli2024quantitative} imply, due to the choice of the coefficients \eqref{definition thetakj_passive_scalar}, that the family $\theta^n$ converges \emph{weakly} to $\overline{\theta}$ solving
\begin{align*}
    \begin{cases}
        \partial_t \overline{\theta}(x,t)&=\kappa_T \Delta \overline{\theta}(t,x)\quad x\in \T^d,\  t\in [0,T]\\
        \overline{\theta}(0,x)&=\overline{\theta}_0.
    \end{cases}
\end{align*}
Our way to reinterpret this result in our framework is the following: let us introduce the family of Young measures
\begin{align*}
    \mu_{n}(t,x,d\theta)=\delta_{\theta^n_t(x)}(d\theta), 
\end{align*}
then the following holds.
\begin{theorem}\label{theorem_passive_scalar}
For every pair $\varphi\in C_b(\R),\psi\in C_b(\T^d)$ and for every $t\in [0,T],\ \overline{\theta}_0\in {H}^1$,  denoting by
\begin{align*}
{\mu}(t,x,d\theta):=  \int_{\mathbb{T}^d}p(\kappa_T t,x-y)\delta_{\overline{\theta}_0(y)}(d\theta) dy,
\end{align*}
then \begin{align}\label{convergence_passive}
    \lim_{n\rightarrow +\infty }\expt{\left\lvert\int_{\T^d\times \R}\varphi(\theta)\psi(x)\mu_n(t,x,d\theta) dx -\int_{\T^d\times \R}\varphi(\theta)\psi(x)\mu(t,x,d\theta) dx \right\rvert^2}=0.
\end{align}
In particular $\mu$ satisfies
\begin{multline}\label{weak_formulation_transport_measure}
    \int_{\T^d\times \R}\varphi(\theta)\psi(x)\mu(t,x,d\theta) dx-\int_{\T^d\times \R}\varphi(\theta)\psi(x)\delta_{\overline{\theta}_0(x)} dx\\=\kappa_T\int_0^t\int_{\T^d\times \R}\varphi(\theta)\Delta\psi(x)\mu(s,x,d\theta) dx ds
\end{multline}
for each $\varphi\in C_b(\R),\ \psi\in C^2_b(\T^d)$
and
\[
\overline{\theta}\left(  x,t\right)  =\int_{\mathbb{R}}\theta{\mu}\left(
t,x,d\theta\right)  .
\]    
\end{theorem}
The limit PDE satisfied by $\mu$ keeps track of the hyperbolic structure of the SPDE satisfied by $\theta^n$ on the contrary that the one for $\overline{\theta}.$ Indeed, choosing formally $\psi\equiv 1,\ \varphi(\theta)=\lvert \theta\rvert^2$
we obtain 
\begin{align*}
    \frac{d}{dt}\int_{\T^d\times \R}\lvert \theta\rvert^2\mu(t,x,d\theta)dx=0.
\end{align*}
By It\^o formula in \autoref{well_posedness_estimates_passive_scalar} we also have
\begin{align*}
    \frac{d}{dt}\int_{\T^d\times \R}\lvert \theta\rvert^2\mu_n(t,x,d\theta)dx=0.
\end{align*}
The relation above fails obviously for the the measure $\delta_{\overline{\theta}_t(x)}(d\theta)dx$. We will come back on this issue in \autoref{remk_no_delta}.
\subsection*{Stochastic Passive Vector Field}
Similarly to previous subsection, let us start introducing the noise affecting our equation for the magnetic field. For each $n\in \N$, $\{W_t^{k,j}\},$ ${{k\in \Z^{3}_0, \  j\in \{1,2\}}}$ we introduce a sequence of complex-valued Brownian motions adapted to $\mathcal{F}_t$ 
    such that
$W^{-k,j}_t=(W^{k,j}_t)^*$ and 
\begin{align*}
    \expt{W^{k,j}_1,{W^{l,m}_1}^*}&=\begin{cases}
        2 & \textit{if } k=l, \ m=j\\
        0 &  \textit{otherwise} ;
    \end{cases}
    \quad \left[W^{k,j}_\cdot, W^{l,m}_\cdot\right]_t=\begin{cases}
        2t & \textit{if } k=-l, \ m=j\\
        0 &  \textit{otherwise }.    
        \end{cases}
\end{align*}
Secondly, for each $ k\in \Z^{3}_0,\ j\in\{1,2\}$ we denote by $\skj(x)=\tkj  a_{k,j}e^{ik\cdot x}$, where $\{\frac{k}{\lvert k\rvert}, a_{k,1}, a_{k,2}\}$ is an orthonormal system of $\R^3$ for $k\in \Gamma_{3,+}$ and $a_{k,j}=a_{-k,j}$ if $k\in \Gamma_{3,-}$, while \begin{align}\label{definition thetakj}
    \tkj &=\sqrt{\chi}\one_{\{n\leq \lvert k\rvert\leq 2n\} }\frac{1}{\lvert k\rvert^{5/2}},\ {\chi}>0. 
\end{align}
With our choices, we ensure that $W^n$ takes values in a space of real valued functions.
It is well known that the family $\{a_{k,j}\frac{1}{2\pi\sqrt{2\pi}}e^{ik\cdot x}\}, \ {{k\in \Z^{3}_0, \ j\in\{1,2\}}}$ is a complete orthogonal systems of $H$ made by eigenfunctions of $-\Delta$.  We set \begin{align}\label{asymptotic_coefficients}
\eta_n=\sum_{\substack{k\in \Z^3_0\\ n\leq \lvert k\rvert \leq 2n}}\frac{1}{\lvert k\rvert^5}=O(n^{-2}),\quad \alpha_n=\sum_{\substack{k\in \Z^3_0\\ n\leq \lvert k\rvert \leq 2n}}\frac{1}{\lvert k\rvert^3}=4\pi\log2+O(n^{-1}).     
\end{align}
Thanks to this choice of the coefficients, the equation for the magnetic field of \autoref{magnetic_field_intro} can be rewritten in It\^o form as 
\begin{align}\label{ito_PDE_magnetic}
 \begin{cases}
     dB^n_t&=\frac{2}{3}\chi\eta_n\Delta B^n_t dt+\sumkj (\skj\cdot\nabla B^n_t-B^n_t\cdot\nabla \skj)dW^{k,j}_t,\\
     \div B^n_t&=0,\\
     B^n_0&=\overline{B}_0,
 \end{cases}   
\end{align}
see \cite[Section 2]{butori2024mean} for details. We are interested to analytically weak solution of \eqref{ito_PDE_magnetic} as described by the following definition.
\begin{definition}\label{weak_solution MF}
A stochastic process  $B^n\in C_{\mathcal{F}}([0,T];H)\cap L^2_{\mathcal{F}}(0,T;V)$ is a weak solution of equation \eqref{ito_PDE_magnetic} if $\mathbb{P}-$a.s.
\begin{align*}
 \langle B^n_t,\phi\rangle-\langle \overline{B}_0,\phi\rangle&=-\frac{2}{3}\chi\eta_n\int_0^t \langle \nabla B^n_s,\nabla \phi\rangle ds+\sumkj \int_0^t \langle\skj \cdot\nabla B^n_s - B^n_s\cdot \nabla \skj ,\phi\rangle dW^{k,j}_s   
\end{align*}
for each $\phi \in V,\ t\in [0,T]$.    
\end{definition}
The well-posedness of \eqref{ito_PDE_magnetic} in the sense of \autoref{weak_solution MF} is guaranteed by the following standard result, which collect also some a priori estimates crucial for the following arguments.
\begin{theorem}\label{well_posedness_estimates_MF}
For each $\overline{B}_0\in V$ there exists a unique weak solution of \eqref{ito_PDE_magnetic} in the sense of \autoref{weak_solution MF}. Moreover $B^n\in C_{\mathcal{F}}(0,T;V_w)$ and the following It\^o formulas hold
\begin{align}
    d\norm{B^n_t}^2&=-2\sumkj \langle B^n_t\cdot\nabla\skj, B^n_t\rangle dW^{k,j}_t+2\sumkj \norm {B^n_t\cdot\nabla\skj}^2 dt,\label{ito2magnetic}\\
    d\norm{B^n_t}^\gamma&=-\gamma \norm{B^n_t}^{\gamma-2}\sumkj \langle B^n_t\cdot\nabla\skj, B^n_t\rangle dW^{k,j}_t+\gamma \norm{B^n_t}^{\gamma-2}\sumkj \norm {B^n_t\cdot\nabla\skj}^2 dt\notag\\ & + \frac{\gamma}{2}\frac{\gamma-2}{2}\norm{B^n_t}^{\gamma-4}\sumkj \lvert\langle B^n_t\cdot\nabla\skj, B^n_t\rangle\rvert^2 dt \quad \gamma\geq 4.\label{itogammamagnetic}
\end{align}
In particular \begin{align}\label{apriori_estimateB_2}
    \expt{\operatorname{sup}_{t\in [0,T]}\norm{B^n_t}^2}&\leq C_0(\overline{B}_0,T,\chi),\\
    \label{apriori_estimate_gamma}
    \expt{\operatorname{sup}_{t\in [0,T]}\norm{B^n_t}^\gamma}&\leq C_1(\overline{B}_0,T,\chi,\gamma).
\end{align}
\end{theorem}
\autoref{well_posedness_estimates_MF} might be well known to the experts. Therefore, we postpone its proof which is based on standard arguments, namely compactness via a vanishing viscosity procedure, to \autoref{appendix PDE magnetic Field}. According to the results of \cite{butori2024mean}, due to the choice of the coefficients \eqref{asymptotic_coefficients}, the family $B^n$ converges \emph{weakly} to the constant profile $\overline{B}_t\equiv\overline{B}_0$. However, looking at the family of Young measures
\begin{align*}
    \mu_{n}(t,x,db)=\delta_{B^n_t(x)}(db), 
\end{align*}
more can be seen. Indeed, denoting by
\begin{align*}
    \mathcal{L}(b):=\frac{8\pi \chi\log 2}{15}\left(2\lvert b\rvert^2 I-b\otimes b\right)
\end{align*}
and by $\mu$ the solution of the following PDE
\begin{align*}
\begin{cases}
\partial_t \mu(t,x,db)&=\operatorname{div}(\mathcal{L}(b)\mu(t,x,db))\quad b\in \R^3,x\in \T^3,\ t\in [0,T]\\
\mu(0,x,db)&=\delta_{\overline{B}_0(x)}(db),
\end{cases}
\end{align*}
defined rigorously in \autoref{sec:magn_field}, then our main result reads as follows:
\begin{theorem}\label{main_thm_magnetic_field}
For each $\overline{B}_0\in V$, $\mu_{n}$ converges in probability in $ C([0,T];\mathcal{Y})$ to ${\mu}$, where $\mu$ is the unique weak solution of \eqref{limit_PDE_vlasov} in the sense of \autoref{weak_vlasov_pde}. Moreover 
\begin{align*}
\overline{B}_0(t,x)=\int_{\R^3}b\mu(t,x,db).
\end{align*}
\end{theorem}
Choosing formally $f(b)=\lvert b\rvert^2$ as test function for the equation satisfied by $\mu$ (this can be done rigorously by approximation as in \autoref{Limi_preserves_moments} below), we obtain, by simple computations,
\begin{align*}
    \int_{\T^3\times \R^3} \lvert b\rvert^2\mu(t,x,db)dx&=\norm{\overline{B}_0}^2+\frac{16\pi \chi\log 2}{3}\int_0^t \int_{\T^3\times \R^3} \lvert b\rvert^2\mu(s,x,db)dx ds.
\end{align*}
Therefore
\begin{align*}
    \int_{\T^3\times \R^3} \lvert b\rvert^2\mu(t,x,db)dx=e^{\frac{16\pi \chi\log 2}{3}t}\norm{\overline{B}_0}^2.
\end{align*}
On the contrary
\begin{align*}
     \int_{\T^3\times \R^3} \lvert b\rvert^2\delta_{\overline{B}_t(x)}(db)dx=\norm{\overline{B}_0}^2.
\end{align*}
We see this fact as a consequence that the limit PDE for $\mu$, on the contrary of the one for $\overline{B}_t$, keeps track of the capability of the magnetic field induced by the motion of a conducting fluid, $B^n_t$, to self excite. This phonemenon is called in the literature magnetic dynamo, \emph{cf. }\cite{moffatt1978magnetic}. 
 We postpone to \autoref{appendix:rmk_physics} some more quantitative comments on the consequences of \autoref{main_thm_magnetic_field} like properties of the support of the measure $\mu(t,x,db)$ and its links with the dynamo theory.

\section{The scalar case}\label{sec:passive_scalar}
We start by observing that, by density, it is enough to prove the validity of equation \autoref{theorem_passive_scalar} for $\varphi\in C^2_b(\R),\ \psi \in C^2_b(\T^d)$. Indeed, assume relation \eqref{convergence_passive} holds for $\varphi\in C^2_b(\R),\ \psi \in C^2_b(\T^d)$. Let us now consider $\varphi\in C_b(\R),\ \psi \in C_b(\T^d)$. For each $\eps>0,$ by density it is possible to find  $\varphi^N\in C^2_b(\R),\ \psi^N \in C^2_b(\T^d)$ such that
\begin{align*}
    \norm{\varphi^N-\varphi}_{C_b}<\sqrt{\eps},\quad  \norm{\psi^N-\psi}_{C_b}<\sqrt{\eps}.
\end{align*}
Therefore
\begin{multline*}
    \expt{\left\lvert\int_{\T^d\times \R}\varphi(\theta)\psi(x)\mu_n(t,x,d\theta) dx -\int_{\T^d\times \R}\varphi(\theta)\psi(x)\mu(t,x,d\theta) dx \right\rvert^2}\\ \lesssim \eps +\expt{\left\lvert\int_{\T^d\times \R}\varphi^N(\theta)\psi^N(x)\mu_n(t,x,d\theta) dx -\int_{\T^d\times \R}\varphi^N(\theta)\psi^N(x)\mu(t,x,d\theta) dx \right\rvert^2}.
\end{multline*}
Considering the limsup in $n$ of the expression above we obtain 
\begin{align*}
    \limsup_{n\rightarrow+\infty }\expt{\left\lvert\int_{\T^d\times \R}\varphi(\theta)\psi(x)\mu_n(t,x,d\theta) dx -\int_{\T^d\times \R}\varphi(\theta)\psi(x)\mu(t,x,d\theta) dx \right\rvert^2}\lesssim \eps
\end{align*}
and the claim follows from the arbitrariness of $\eps$.
Secondly, one can prove that the solutions given by \autoref{well_posedness_estimates_passive_scalar} are renormalizable. By
this we mean that, given $\varphi\in C^2_b(\R),\ \psi \in C^2_b(\T^d)$, one has
\begin{align*}
    \varphi\left(  \theta^n_t  \right)=\int_{\R}\varphi(\theta)d\delta_{\theta^n_t(\cdot)}\in C_{\mathcal{F}}([0,T];H^1_w(\T^d))\cap C_{\mathcal{F}}([0,T];L^2(\T^d))
\end{align*}
and
\begin{align}\label{vlasov_passive_stoch}
    \int_{\T^d} \varphi(\theta^n_t(x))\psi(x) dx&=\int_{\T^d} \varphi(\overline{\theta}_0(x))\psi(x) dx+\kappa_T\int_0^t\int_{\T^d} \varphi({\theta}_s(x))\Delta\psi(x) dxds\notag\\ & -\sum_{\substack{k\in \Z^d_0\\
    j\in\{1,\dots,d\}}}\int_0^t\int_{\T^d} \varphi({\theta}_s(x))\skj\cdot\nabla\psi(x)dW^{k,j}_s\quad \mathbb{P}-a.s.
\end{align}
In particular, by \autoref{well_posedness_estimates_passive_scalar}, it is the unique solution of \eqref{ito_PDE_passive_scalar} in the sense of \autoref{passive_scalar_weak} with initial condition $\varphi\circ \overline{\theta}_0.$
The proof of this claim can be obtained following verbatim the argument provided below in \autoref{a_priori_magnetic} in a more complex framework for the magnetic field, therefore we avoid to add details in this easier framework.\\
Recalling the definition of the  Young measure
\begin{align*}
\mu_{n}\left(t,x,d\theta  \right)    & =\delta_{\theta^n_t\left(  x\right)
}\left(  d\theta\right),
\end{align*}
we can rewrite \eqref{vlasov_passive_stoch} as 
\begin{align*}
\int_{\T^d\times \R} \varphi(\theta)\psi(x)\mu_n(t,x,db)dx= &\int_{\T^d\times \R} \varphi(\theta)\psi(x)\delta_{\overline{\theta}_0(x)}(db)dx\\ &+\kappa_T\int_0^t \int_{\T^d\times \R} \varphi(\theta)\Delta\psi(x)\mu_n(s,x,db)dxds\\ &-\sum_{\substack{k\in \Z_0^d\\ j\in\{1,\dots, d\}}}\int_0^t \int_{\T^d\times \R} \varphi(\theta)\skj(x)\cdot\nabla\psi(x)\mu_n(s,x,db)dxds\ dW^{k,j}_s.
\end{align*}
Notice that we cannot deduce any convergence of $\int_{\T^d\times \R} \varphi(\theta)\psi(x)\mu_n(t,x,db)dx$ from
the weak convergence of $\theta^n$ guaranteed by the results of \cite{galeati2020convergence}, \cite{flandoli2024quantitative}. Thus, the equations for $\theta^n$ lead
to equations for $\mu_{n}$, but convergence properties do not translate
immediately from $\theta^n$ to $\mu_{n}$.

The following remark may help to get an intuition about the main result of the section.

\begin{remark}
\label{rem2}Assume $\overline{\theta}_0\left(  x\right)  =\sum_{j}a_{j}1_{D_{j}}\left(
x\right)  $ over a finite partition $\left(  D_{j}\right)  $, with different
values $a_{j}$. Then%
\[
\theta^n\left(  x,t\right)  =\sum_{j}a_{j}1_{D_{j}^{n}\left(  t\right)
}\left(  x\right)
\]
where $\left\{  D_{j}^{n}\left(  t\right)  \right\}  $ is again a (random)
partition, and
\[
\mu_{n}\left( t,x,d\theta\right)  =\delta_{\theta^n_t\left(  x\right)
}\left(  d\theta\right)  =\sum_{j}\delta_{a_{j}}\left(  d\theta\right)
1_{D_{j}^{n}\left(  t\right)  }\left(  x\right)  .
\]
The process $1_{D_{j}^{n}\left(  t\right)  }\left(  x\right)  $ is the
solution of the stochastic transport equation with initial condition
$1_{D_{j}}\left(  x\right)  $. Applying for every $j$ the results above we get
that $\mu_{n}\left(  t,x,d\theta\right)  $ weakly converges to%
\[
{\mu}\left( t,x,d\theta\right)  =\sum_{j}\delta_{a_{j}}\left(
d\theta\right)  \overline{\theta}_{j}\left(  x,t\right)
\]
where $\overline{\theta}_{j}\left(  x,t\right)  $ is the solution of%
\begin{align*}
\begin{cases}
    \partial_{t}\overline{\theta}_{j}  & =\kappa_{T}\Delta\overline{\theta}_{j}\\
\overline{\theta}_{j}|_{t=0}  & =1_{D_{j}}
\end{cases}
\end{align*}
\end{remark}
Now we are ready to provide the proof of \autoref{theorem_passive_scalar}.
\begin{proof}[Proof of \autoref{theorem_passive_scalar}]
Due to the discussion above it is enough to show the theorem for $\varphi\in~ C^2_b(\R),$ $\psi\in~ C^2_b(\T^d).$ Moreover we know that, denoting by $\varphi^n_t(x)=\varphi(\theta^n_t(x))$, it  solves
\begin{align*}
\begin{cases}
    d\varphi^n_t&=\kappa_T\Delta \varphi^n_tdt+\sum_{\substack{k\in \Z^d_0\\ j\in\{1,\dots,d\}}}\skj \cdot\nabla \varphi^n_t dW^{k,j}_t\\
    \varphi^n_0&=\varphi\circ\overline{\theta}_0 
\end{cases}  
\end{align*}
in the sense of \autoref{passive_scalar_weak}. Due to our choice of $\varphi$, $\sup_{t\in [0,T]}, \norm{\varphi^n_t}_{L^{\infty}}\leq \norm{\varphi}_{C_b}\ \mathbb{P}-a.s.$
Due to our choice of the coefficients $\skj$, see in particular equation \eqref{definition thetakj_passive_scalar}, we can apply \cite[Theorem 1.7]{flandoli2024quantitative}. If we denote by $\overline{\varphi}_{\overline{\theta}_0}$ the unique weak solution of
\begin{align*}
\begin{cases}
\partial_t\overline{\varphi}_{\overline{\theta}_0}&=\kappa_T\Delta \overline{\varphi}_{\overline{\theta}_0}\\
    \overline{\varphi}_{\overline{\theta}_0}(0)&=\varphi\circ\overline{\theta}_0,
\end{cases}  
\end{align*}    
for each $\psi$ we have
\begin{align*}
    \mathbb{E}\left[\langle \overline{\varphi}_{\overline{\theta}_0}(t)-\varphi^n(t),\psi\rangle^2 \right]\lesssim \frac{\norm{\psi}^2\norm{\kappa_T\varphi}_{C_b}^2}{n^d}\rightarrow 0\quad\text{as }n\rightarrow +\infty.
\end{align*}
Let us rewrite $\langle \overline{\varphi}_{\theta_0}(t),\psi\rangle$ and $\langle \varphi^n(t),\psi\rangle$:
\begin{align*}
\langle \varphi^n(t),\psi\rangle&=\int_{\T^d} \psi(x)\varphi\circ \theta^n_t(x) dx\\ & = \int_{\T^d} \psi(x)\int_{\mathbb{R}}\varphi(\theta)\delta_{\theta^n_t(x)}(d\theta) dx\\ & =  \int_{\T^d\times \mathbb{R}} \psi(x)\varphi(\theta)\mu_n(t,x,d\theta) dx 
\end{align*}
recalling the definition of $\mu_n(t,x,d\theta)$.
While
\begin{align*}
\langle \overline{\varphi}_{\overline{\theta}_0},\psi\rangle &=\int_{\mathbb{T}^d} \psi(x)\int_{\mathbb{T}^d} p(\kappa_Tt,x-y)\varphi\circ \overline{\theta}_0(y) dy dx\\ &=\int_{\mathbb{T}^d} \psi(x)\int_{\mathbb{T}^d}\int_{\mathbb{R}} p(\kappa_Tt,x-y)\varphi(\theta) \delta_{\overline{\theta}_0(y)}(d\theta) dy dx\\ & =\int_{\mathbb{T}^d} \psi(x)\int_{\mathbb{R}}\varphi(\theta)\int_{\mathbb{T}^d}p(\kappa_Tt,x-y)\delta_{\overline{\theta}_0(y)}(d\theta) dy dx\\ &= \int_{\mathbb{T}^d} \psi(x)\int_{\mathbb{R}}\varphi(\theta){\mu}(t,x,d\theta) dx,
\end{align*}
where
\begin{align*}
{\mu}(t,x,d\theta):=  \int_{\mathbb{T}^d}p(\kappa_Tt,x-y)\delta_{\overline{\theta}_0(y)}(d\theta) dy.  
\end{align*}
The remaining claims follow from the definition above of $\mu$ and the one of heat kernel. This concludes the proof.
\end{proof}
\begin{remark}\label{remk_no_delta}
The measure%
\[
\overline{\mu}\left(  t,x,d\theta\right)  =\delta_{\overline{\theta}_t\left(
x\right)  }\left(  d\theta\right)
\]
is \textit{not} a solution of the limit PDE. Indeed it satisfies%
\begin{align*}
\int_{\T^d\times \R} \psi(x)\varphi(\theta)\overline{\mu}\left(  t,x,d\theta\right) dx  & =\int_{\T^d}\psi\left(  x\right)  \varphi\left(
\overline{\theta}_t\left(  x\right)  \right)  dx\\
\int_{\T^d\times \R} \psi(x)\varphi(\theta)\overline{\mu}\left(  0,x,d\theta\right) dx  & =\int_{\T^d}\psi\left(  x\right)  \varphi\left(
\overline{\theta}_0\left(  x\right)  \right)  dx\\
\kappa_{T}\int_{0}^{t}\int_{\T^d\times \R}\Delta\psi(x)\varphi(\theta)\overline{\mu}\left(  s,x,d\theta\right)dxds  & =\kappa_{T}\int_{0}^{t}\int_{\T^d}\Delta\psi\left(  x\right)  \varphi\left(
\overline{\theta}_s\left(  x\right)  \right)  dxds
\end{align*}
and
\begin{align*}
\partial_{t}\varphi\left(  \overline{\theta}_t\left(  x\right)  \right)    &
=\varphi^{\prime}\left(  \overline{\theta}_t\left(  x\right)  \right)
\partial_{t}\overline{\theta}_t\left(  x\right)  \\
\nabla\varphi\left(  \overline{\theta}_t\left(  x\right)  \right)    &
=\varphi^{\prime}\left(  \overline{\theta}_t\left(  x\right)  \right)
\nabla\overline{\theta}_t\left(  x\right),
\end{align*}
but%
\begin{align*}
\Delta\varphi\left(  \overline{\theta}_t\left(  t\right)  \right)    &
=\varphi^{\prime\prime}\left(  \overline{\theta}_t\left(  x\right)  \right)
\left\vert \nabla\overline{\theta}_t\left(  x\right)  \right\vert ^{2}%
+\varphi^{\prime}\left(  \overline{\theta}_t\left(  x\right)  \right)
\Delta\overline{\theta}_t\left(  x\right)  \\
& \neq\varphi^{\prime}\left(  \overline{\theta}_t\left(  x\right)  \right)
\Delta\overline{\theta}_t\left(  x\right)  .
\end{align*} 
Moreover, let us observe the following fact concerning the different behaviour of the limit measure $\mu(t,x,d\theta)$ and $\overline{\mu}(t,x,d\theta)$. Let us choose as test functions for relation \eqref{weak_formulation_transport_measure} $\psi\equiv 1$ and a family $\varphi^M:\R\rightarrow \R$ defined as $\varphi^M(\theta)=M\wedge \lvert \theta\rvert^2$. In this way we obtain
\begin{align*}
    \int_{\T^d\times \R}\varphi^M(\theta)\mu(t,x,d\theta)dx=\int_{\T^d}M\wedge \lvert \overline{\theta}_0(x)\rvert^2 dx.
\end{align*}
Therefore, letting $M\rightarrow +\infty,$ due to monotone convergence theorem
\begin{align*}
 \int_{\T^d\times \R}  \lvert \theta\rvert^2\mu(t,x,d\theta)dx=\norm{\overline{\theta}_0}^2.     
\end{align*}
Similarly, the It\^o formula of  \autoref{well_posedness_estimates_passive_scalar} gives 
\begin{align*}
 \int_{\T^d\times \R}  \lvert \theta\rvert^2\mu_n(t,x,d\theta)dx=\norm{\overline{\theta}_0}^2.     
\end{align*}
On the contrary, the dissipative nature of the equation satisfied by $\overline{\theta}$ implies
\begin{align*}
    \int_{\T^d\times \R}  \lvert \theta\rvert^2\overline{\mu}(t,x,d\theta)dx\leq \norm{\overline{\theta}_0}^2e^{-2\kappa_T t}.
\end{align*}
\end{remark}

\section{The Case of the Magnetic Field}\label{sec:magn_field}
The proof of \autoref{main_thm_magnetic_field} relies on a stochastic compactness argument on $\mu_n(t,x,db)$. The proof is splitted in three main steps corresponding to \autoref{a_priori_magnetic}, \autoref{compactness_stochastic_young_measures} and \autoref{passage_to_the_limit_sec}. 
We prove some properties of the Young measure $\mu_n(t,x,db)$ in \autoref{a_priori_magnetic}. Such properties allow us to show the tightness of the laws of $\mu_n(t,x,db)$ in the space $C([0,T];\mathcal{Y})$ in \autoref{compactness_stochastic_young_measures}. Thanks Skorokhod's representation theorem, we can pass to an auxiliary probability space where the convergence in law becomes $\mathbb{P}-a.s.$ convergence in \autoref{passage_to_the_limit_sec}. This allows us to identify a limit Young measure which satisfies a deterministic PDE and conclude the proof by uniqueness of the solutions of this PDE in our class.\\ In order to ease the notation, let us fix $\chi=1$ in the following.
\subsection{Properties of the PDE for the Magnetic Field}\label{a_priori_magnetic}
In this section we collect some useful properties of the Young measure associated to $B^n$ solution of \eqref{ito_PDE_magnetic}: we will introduce the Vlasov type SPDE satisfied by $\mu_n(t,x,db)=\delta_{B^n_t(x)}(db)$ and then we will prove some estimates on $\mu_n(t,x,db)$, uniformly in the scaling parameter $n$.
\subsection*{Vlasov type SPDE for $\mu_n$}
In this section we show that the measure $\mu_n(t,x,db)=\delta_{B^n_t(x)}(db)$ satisfied the Vlasov type PDE  
\begin{align}\label{Vlasov PDE}
\begin{cases}
    d\mu_n&=  \sumkj \div_b(b\cdot\nabla \skj\mu_n)  dW^{k,j}_t+\sumkj \div_x(\skj \mu_n) dW^{k,j}_t\\ & +\operatorname{div}_b(\mathcal{L}^n(b)\nabla_b\mu_n)dt +\frac{2}{3}\eta_n\Delta_x \mu_{n}dt,\\
    \mu_n(0,x,db)&=\delta_{\overline{B}_0(x)}(db),
\end{cases} 
\end{align}
where
\begin{align*}
    \mathcal{L}^n(b)=\sum_{\substack{k\in \Z^3_0\\ n\leq \lvert k\rvert \leq 2n}}\frac{(b\cdot k)^2}{\lvert k\rvert^5}\left(I-\frac{k\otimes k}{\lvert k\rvert^2}\right),
\end{align*}
in the sense that for each $f\in C^{2}_b(\T^3\times \R^3)$ 
\begin{multline}\label{vlasov_weak}
    \int_{\T^3\times \R^3}f(x,b)\mu_n(t,x,db) dx-\int_{\T^3\times \R^3}f(x,b)\delta_{\overline{B}_0(x)}(db)dx\\ =-\sumkj\int_0^t\int_{\T^3\times \R^3} \left(\skj(x)\cdot(\nabla_x f)(x,b)+b\cdot\nabla\skj(x)\nabla_b f(x,b)\right)\mu_n(s,x,db)dxdW^{k,j}_s \\  +\frac{2}{3}\eta_n\int_0^t\int_{\T^3\times \R^3} \Delta_x f(x,b)\mu_n(s,x,db) dx ds+\int_0^t\int_{\T^3\times \R^3} \operatorname{div}_b\left(\mathcal{L}^n(b)\nabla_b f(x,b)\right)\mu_n(s,x,db) dx ds
\end{multline}
$\mathbb{P}-a.s.$ for each $t\in [0,T].$ 
\begin{remark}\label{alter_vlasov_weak}
Due to the definition of $\mu_n(t,x,db)$, relation \eqref{vlasov_weak} can be rewritten as 
\begin{align*}
    \int_{\T^3}f(x,B^n_t(x))-f(x,\overline{B}_0(x)) dx&=-\sumkj\int_0^t\int_{\T^3} \skj(x)\cdot\left(\nabla_x f\right)(x,B^n_s(x))dW^{k,j}_s\\ & -\sumkj\int_0^t\int_{\T^3}B^n_s(x)\cdot\nabla\skj(x)\nabla_b f(x,B^n_s(x))dW^{k,j}_s \notag\\ & +\frac{2}{3}\eta_n\int_0^t\int_{\T^3} \left(\Delta_x f\right)(x,B^n_s(x)) dxds\\ & +\int_0^t\int_{\T^3} \operatorname{div}_b\left(\mathcal{L}^n(B^n_s(x))\nabla_b f(x,B^n_s(x))\right) dx ds\quad \mathbb{P}-a.s. \ \forall t\in [0,T].
\end{align*}
We will actually prove this relation, obtaining \eqref{vlasov_weak} thanks to the definition of $\mu_n$.
\end{remark}
Let $\varrho\in C^{\infty}(\R^3)$ a non negative, even function such that 
\begin{align*}
    \int_{\R^3}\varrho(x)dx=1,\quad \operatorname{supp}\varrho\subseteq [-1/2,1/2]^3.
\end{align*}
For $\eps>0$, let $\varrho_{\eps}(\cdot)=\frac{1}{\eps^3}\varrho(\cdot/\eps).$ If $f\in L^1(\T^3)$ it can be extended to a locally integrable periodic function on $\R^3$, so that the convolution $\varrho_{\eps}\ast f$ makes sense and it is still a periodic function. It is easy to show, see for example \cite[Remark 2.15]{butori2024mean} that if $\phi\in C^{\infty}(\T^3)$ then
\begin{align*}
    \langle B^n_t,\phi\rangle&=\langle B^n_0,\phi\rangle+\frac{2}{3}\eta_n\int_0^t \langle B^n_s,\Delta \phi \rangle ds-\sumkj \int_0^t \langle \skj\cdot\nabla \phi,B^n_s\rangle dW^{k,j}_s\\ & -\sumkj \int_0^t \langle B^n_s\cdot\nabla\skj,\phi\rangle dW^{k,j}_s.
\end{align*}
In particular, denoting by $B^{n,\eps}=B^n\ast \varrho_{\eps}$ and choosing $\phi^x(y)=\left(\varrho_{\eps}(x-y),\varrho_{\eps}(x-y),\varrho_{\eps}(x-y)\right)^t$ as a test function we obtain
\begin{align}\label{regularized_PDE}
    B^{n,\eps}_t(x)&=B^{n,\eps}_0(x)+\frac{2}{3}\eta_n \int_0^t\Delta B^{n,\eps}_s(x) ds\notag\\ &+\sumkj\int_0^t\left(\skj(x)\cdot\nabla B^{n,\eps}_s(x)-B^{n,\eps}_s(x)\cdot\nabla \skj(x)\right)dW^{k,j}_s\notag\\ & -\sumkj
    \int_0^t\int_{\R^3} \left(\skj(x)-\skj(x-z)\right)\cdot \nabla\varrho_{\eps}(z)B^n_s(x-z)  dz dW^{k,j}_s\notag\\ & -\sumkj
    \int_0^t\int_{\R^3} B^n_s(z)\cdot\nabla\left(\skj(z)-\skj(x)\right)\varrho_{\eps}(x-z) dz dW^{k,j}_s.
\end{align}
Now on for this subsection we will drop the index $n$ in our notation, since we are interested to prove the validity of \eqref{vlasov_weak} for $n$ fixed. In order to reach our goal, as it is customary in the framework of hyperbolic equations, we rely on a commutator lemma.
Therefore, let us introduce a notation we will use in the following.
If $v\in C^{\infty}(\T^3;\R^3),\ B\in L^p(\T^3;\R^3)$ are divergence free vector field for some $p\in [1,+\infty)$ we denote by
\begin{align*}
    [v\cdot\nabla,\varrho_{\eps}]B:=\int_{\R^3} \left(v(x)-v(x-z)\right)\cdot \nabla\varrho_{\eps}(z)B(x-z)  dz
\end{align*}
and by
\begin{align*}
    [\cdot\nabla v,\varrho_{\eps}]B:=\int_{\R^3} B(z)\cdot\nabla\left(v(z)-v(x)\right)\varrho_{\eps}(x-z) dz. 
\end{align*}
\begin{lemma}\label{commutator_lemma}
Let $p\in [1,+\infty)$,\ $v\in C^{\infty}(\T^3;\R^3),\ B\in L^p(\T^3;\R^3)$ divergence free vector fields. 
Then
\begin{align*}
    \norm{[v\cdot\nabla,\varrho_{\eps}]B}_{L^p(\T^3;\R^3)}\rightarrow 0,\quad \norm{[\cdot\nabla v,\varrho_{\eps}]B}_{L^p(\T^3;\R^3)}\rightarrow 0\quad \text{as }\eps\rightarrow 0.
\end{align*}
Moreover
\begin{align*}
 \norm{[v\cdot\nabla,\varrho_{\eps}]B}_{L^p(\T^3;\R^3)}\lesssim \norm{v}_{C^1_b}\norm{\varrho}_{C^1_b} \norm{B}_{L^p},\quad \norm{[\cdot\nabla v,\varrho_{\eps}]B}_{L^p(\T^3;\R^3)}\lesssim \eps\norm{v}_{C^2_b}\norm{\varrho}_{C^0_b}\norm{B}_{L^p}. 
\end{align*}
\end{lemma}
\begin{proof}
First we observe that by change of variables
\begin{align*}
[v\cdot\nabla,\varrho_{\eps}]B(x)&=\int_{\R^3}\frac{v(x)-v(x-\eps y)}{\eps}\nabla \varrho(y)B(x-\eps y) dy.
\end{align*}
Secondly, since $v$ is smooth, the fundamental theorem of calculus implies
\begin{align*}
    v(x)-v(x-\eps y)=\eps\int_0^1 y\cdot\nabla v(x-(1-s)y) ds.
\end{align*}
Therefore
\begin{align}\label{equation_commutator}
[v\cdot\nabla,\varrho_{\eps}]B(x)=\int_{\R^3}\int_0^1 y\cdot \nabla v(x-(1-s)\eps y)\nabla \varrho(y) B(x-\eps y)ds dy.
\end{align}
Now we are ready to compute the $L^p$ norm of the commutator. Indeed, by H\"older's inequality, since the support of $\varrho$ is compact
\begin{align*}
    \int_{\T^3}\left\lvert [v\cdot\nabla,\varrho_{\eps}]B(x)\right\rvert^p dx& \lesssim \norm{v}_{C^1_b}^p\norm{\varrho}_{C^1_b}^p\int_{\T^3}\int_{\operatorname{supp} \varrho}\lvert B(x-\eps y)\rvert^p dx dy\\ & \lesssim  \norm{v}_{C^1_b}^p\norm{\varrho}_{C^1_b}^p \norm{B}_{L^p}^p.
\end{align*}
Concerning the convergence to $0$ we observe that, since 
\begin{align*}
    B(x)\int_{\R^3} y\cdot \nabla v(x)\nabla \varrho(y) dy&=-\sum_{i,j=1}^3 B(x)\int_{\R^3} \delta_{i,j} \partial_j v^j(x)\varrho(y) dy\\ &\equiv 0,
\end{align*}
it is enough to show that
\begin{align*}
    \norm{[v\cdot\nabla,\varrho_{\eps}]B-B(\cdot)\int_{\R^3} y\cdot \nabla v(\cdot)\nabla \varrho(y) dy}_{L^p}\rightarrow 0.
\end{align*}
Thanks to H\"older inequality, relation \eqref{equation_commutator} and the fact that the support of $\varrho$ is included in $[-1/2,1/2]^3$ we obtain
\begin{multline*}
  \norm{[v\cdot\nabla,\varrho_{\eps}]B-B(\cdot)\int_{\R^3} y\cdot \nabla v(\cdot)\nabla \varrho(y) dy}_{L^p}^p\\ \lesssim  \int_{[-1/2,1/2]^3}\int_0^1 \int_{\T^3} \lvert B(x)\nabla v(x)-B(x-\eps y)\nabla v(x-(1-s)\eps y) \rvert^p dx ds dy\\ = \int_{[-1/2,1/2]^3} f^{\eps}(y) dy,
\end{multline*}
where we denoted by 
\begin{align*}
 f^{\eps}(y)=\int_0^1 \int_{\T^3} \lvert B(x)\nabla v(x)-B(x-\eps y)\nabla v(x-(1-s)\eps y) \rvert^p dx ds.   
\end{align*}
The function $f^{\eps}\rightarrow 0$ a.e. indeed
\begin{align*}
  \lvert B(x)\nabla v(x)-B(x-\eps y)\nabla v(x-(1-s)\eps y) \rvert^p& \lesssim  \lvert B(x)\nabla v(x)-B(x-\eps y)\nabla v(x-\eps y) \rvert^p  \\ &+ \lvert B(x-\eps y)\rvert^p\lvert\nabla v(x-\eps y) -\nabla v(x-(1-s)\eps y) \rvert^p\\ & \leq   \lvert B(x)\nabla v(x)-B(x-\eps y)\nabla v(x-\eps y) \rvert^p  \\ &+ \eps^p \lvert B(x-\eps y)\rvert^p\norm{v}_{C^2_b}^p.
\end{align*}
Therefore 
\begin{align*}
    f^{\eps}(y)\lesssim \int_{0}^1 \int_{\T^3}\lvert B(x)\nabla v(x)-B(x-\eps y)\nabla v(x-\eps y) \rvert^p dx ds+\eps^p\norm{v}_{C^2_b}^p \norm{B}_{L^p}^p\rightarrow 0
\end{align*}
since the first term converges to $0$ thanks to the continuity of translations in $L^p$. Moreover, we can apply dominated convergence theorem since previous computations imply by change of variables 
\begin{align*}
f^{\eps}(y)\lesssim  \norm{v}_{C^2_b}^p \norm{B}_{L^p}^p.    
\end{align*}
The analysis of the second commutator is easier. Indeed, by change of variables
\begin{align*}
    \left\lvert\int_{\R^3} B(z)\cdot\nabla\left(v(z)-v(x)\right)\varrho_{\eps}(x-z) dz\right\rvert&=\left\lvert\int_{\R^3} B(x-\eps z)\cdot\nabla\left(v(x-\eps z)-v(x)\right)\varrho(z) dz\right\rvert\\ & \leq \eps \norm{v}_{C^2_b}\norm{\varrho}_{C^0_b}\int_{[-1/2,1/2]^3}\lvert B(x-\eps z)\rvert dz.
\end{align*}
Therefore by H\"older inequality
\begin{align*}
  \norm{[\cdot\nabla v,\varrho_{\eps}]B}_{L^p(\T^3;\R^3)}^p& \lesssim \eps^p\norm{v}^p_{C^2_b}\norm{\varrho}^p_{C^0_b}\norm{B}_{L^p}^p\rightarrow 0.  
\end{align*}
\end{proof}
Next step is to compute, for a smooth function $f\in C^{2}_b(\T^3\times \R^3)$ the evolution of $\int_{\T^3} f(x,B^{\eps}_t)dx.$ This is the content of the following lemma.
\begin{lemma}\label{regularized vlasov weak}
For each $f\in C^{2}_b(\T^3\times \R^3)$ it holds
\begin{align*}
\int_{\T^3} f(x,B^{\eps}_t(x))dx&= \int_{\T^3} f(x,B^{\eps}_0(x))dx+\frac{2}{3}\eta\int_0^t\int_{\T^3} \left(\Delta_x f\right)(x,B^\eps_s(x)) dxds\\
&+\int_0^t\int_{\T^3} \operatorname{div}_b\left(\mathcal{L}(B^{\eps}_s(x))\nabla_b f(x,B^\eps_s(x))\right) dx ds
\\ &-\sumkj\int_0^t\int_{\T^3}B^\eps_s(x)\cdot\nabla\sigma_{k,j}(x)\nabla_b f(x,B^\eps_s(x))dx dW^{k,j}_s \\ & - \sumkj\int_0^t\int_{\T^3}\sigma_{k,j}(x)\cdot(\nabla_x f)(x,B^\eps_s(x))dx dW^{k,j}_s\\ & -\sumkj \int_0^t \int_{\T^3}\left([\sigma_{k,j}\cdot\nabla, \varrho_{\eps}]+[\cdot\nabla \sigma_{k,j}, \varrho_{\eps} ]\right)B_s(x)\cdot\nabla_b f(x,B^{\eps}_s) dx dW^{k,j}_s \\ & +\sumkj \int_0^t\int_{\T^3}Tr\left(\nabla^2_b f(x,B^{\eps}_s(x))\left(\left([\sigma_{k,j}\cdot\nabla, \varrho_{\eps}]+[\cdot\nabla \sigma_{k,j}, \varrho_{\eps} ]\right)B_s(x)\right.\right.  \\ & \quad \quad \quad \quad \quad\quad \quad \quad \quad  \otimes \left.\left.\left([\sigma_{-k,j}\cdot\nabla, \varrho_{\eps}]+[\cdot\nabla \sigma_{-k,j}, \varrho_{\eps} ]\right)B_s(x)\right)\right) dx ds\\ & -\sumkj \int_0^t\int_{\T^3}Tr\left(\nabla^2_b f(x,B^{\eps}_s(x))\left(\left([\sigma_{k,j}\cdot\nabla, \varrho_{\eps}]+[\cdot\nabla \sigma_{k,j}, \varrho_{\eps} ]\right)B_s(x)\right.\right.  \\ & \quad \quad \quad \quad \quad\quad \quad \quad \quad  \otimes \left.\left. \left(\sigma_{-k,j}(x)\cdot \nabla B^{\eps}_s(x)-B^{\eps}_s(x)\cdot\nabla \sigma_{-k,j}(x)\right)\right)\right) dx ds\\ & -\sumkj \int_0^t\int_{\T^3}Tr\left(\nabla^2_b f(x,B^{\eps}_s(x))\left(\left(\sigma_{k,j}(x)\cdot \nabla B^{\eps}_s(x)-B^{\eps}_s(x)\cdot\nabla \sigma_{k,j}(x)\right)\right.\right.  \\ & \quad \quad \quad \quad \quad\quad \quad \quad \quad  \otimes \left.\left.\left([\sigma_{-k,j}\cdot\nabla, \varrho_{\eps}]+[\cdot\nabla \sigma_{-k,j}, \varrho_{\eps} ]\right)B_s(x)\right)\right) dx ds.
\end{align*}
\end{lemma}
\begin{proof}
Since $B^{\eps}$ is a smooth object in space which satisfies in a strong sense relation \eqref{regularized_PDE}, we can apply finite dimensional It\^o formula obtaining
\begin{align}\label{vlasov_pde_step_1}
f(x,B^{\eps}_t(x))&=f(x,B^{\eps}_0(x))+\frac{2}{3}\eta\int_0^t \nabla_b f(x,B^{\eps}_s(x))\cdot \Delta B^{\eps}_s(x)ds+\sumkj\int_0^t R^{k,j,\eps,f}_s(x) ds\notag\\ & +\sumkj\int_0^t  \nabla_b f(x,B^{\eps}_s(x)) \cdot\left(\sigma_{k,j}(x)\cdot\nabla B^{\eps}_s(x)-B^{\eps}_s(x)\cdot\nabla\sigma_{k,j}(x)\right)dW^{k,j}\notag\\ & -\sumkj\int_0^t  \nabla_b f(x,B^{\eps}_s(x)) \cdot\left([\sigma_{k,j}\cdot\nabla,\varrho_{\eps}]+[\cdot\nabla\sigma_{k,j},\varrho_{\eps}]\right) B_s(x)dW^{k,j},
\end{align}
where \begin{align*}
&R^{k,j,\eps,f}_s(\cdot)\\ &=Tr(\nabla^2_b f(\cdot,B^{\eps}_s(\cdot))\left(\sigma_{k,j}(\cdot)\cdot\nabla B^{\eps}_s(\cdot)-B^{\eps}_s(\cdot)\cdot\nabla\sigma_{k,j}(\cdot)-\left([\sigma_{k,j}\cdot\nabla,\varrho_{\eps}]+[\cdot\nabla\sigma_{k,j},\varrho_{\eps}]\right) B_s(\cdot)\right.\\ & \qquad \qquad\quad\quad\left.\otimes \sigma_{-k,j}(\cdot)\cdot\nabla B^{\eps}_s(\cdot)-B^{\eps}_s(\cdot)\cdot\nabla\sigma_{-k,j}(\cdot)-\left([\sigma_{-k,j}\cdot\nabla,\varrho_{\eps}]+[\cdot\nabla\sigma_{-k,j},\varrho_{\eps}]\right) B_s(\cdot) \right).
\end{align*}
In order to get the claim integrating \eqref{vlasov_pde_step_1} on $\T^3$ we need to make some observations.
First let us observe that by chain rule and integration by parts
\begin{align*}
 &\int_{\T^3}   \nabla_b f(x,B^{\eps}_s(x))\cdot\left(\sigma_{k,j}(x)\cdot\nabla B^{\eps}_s(x)\right)dx\\ &=\int_{\T^3}   \nabla_x f(x,B^{\eps}_s(x))\cdot\sigma_{k,j}(x)- (\nabla_x f)(x,B^{\eps}_s(x))\cdot\sigma_{k,j}(x)dx\\ & =-\int_{\T^3}  (\nabla_x f)(x,B^{\eps}_s(x))\cdot\sigma_{k,j}(x)dx.
\end{align*}
Therefore
\begin{multline}\label{vlasov_pde_step_2}
\sumkj\int_0^t \int_{\T^3} \nabla_b f(x,B^{\eps}_s(x)) \cdot\sigma_{k,j}(x)\cdot\nabla B^{\eps}_s(x)dxdW^{k,j}\\ =-  \sumkj\int_0^t \int_{\T^3}  (\nabla_x f)(x,B^{\eps}_s(x))\cdot\sigma_{k,j}(x)dxdW^{k,j}.  
\end{multline}  
Secondly, let us observe that for each $x\in \T^3,\ t\in [0,T]$
\begin{align*}
    (\sigma_{k,j}\cdot\nabla B^{\eps})\otimes (B^{\eps}\cdot\nabla\sigma_{-k,j})+(\sigma_{-k,j}\cdot\nabla B^{\eps})\otimes (B^{\eps}\cdot\nabla\sigma_{k,j})&=0,\\
    (B^{\eps}\cdot\nabla\sigma_{k,j})\otimes(\sigma_{-k,j}\cdot\nabla B^{\eps})+(B^{\eps}\cdot\nabla\sigma_{-k,j})\otimes(\sigma_{k,j}\cdot\nabla B^{\eps})&=0.
\end{align*}
Therefore
\begin{multline}\label{vlasov_pde_step_3}
   \sumkj\int_0^t \int_{\T^3}Tr\left(\nabla^2_b f(x,B^{\eps}_s(x))(\sigma_{k,j}(x)\cdot\nabla B^{\eps}_s(x))\otimes (B^{\eps}_s(x)\cdot\nabla\sigma_{-k,j}(x))\right) dx ds\\
   +\sumkj\int_0^t \int_{\T^3} Tr\left(\nabla^2_bf(x,B^{\eps}_s(x))(B^{\eps}_s(x)\cdot\nabla\sigma_{k,j}(x))\otimes (\sigma_{-k,j}(x)\cdot\nabla B^{\eps}_s(x))\right) dx ds=0.
\end{multline}
Considering $\sumkj \int_0^t\int_{\T^3}Tr\left(\nabla^2_b f(x,B^{\eps}_s(x))(B^{\eps}_s(x)\cdot\nabla\sigma_{k,j}(x))\otimes (B^{\eps}_s(x)\cdot\nabla\sigma_{-k,j}(x))\right)dxds $, we observe 
that for each $x\in \T^3,\ t\in [0,T]$
\begin{multline*}
    \sumkj Tr\left(\nabla^2_b f(\cdot,B^{\eps})(B^{\eps}\cdot\nabla\sigma_{k,j}\otimes B^{\eps}\cdot\nabla\sigma_{-k,j}\right)\\ =\sum_{\substack{k\in \Z^3_0\\ n\leq \lvert k\rvert \leq 2n}}\sum_{\alpha,\beta=1}^3\frac{(B^{\eps}\cdot k)^2}{\lvert k\rvert^5} \left(I-\frac{k\otimes k}{\lvert k\rvert^2}\right)_{\alpha,\beta}\partial_{b,\alpha,\beta} f(\cdot, B^{\eps})
\end{multline*}
and by definition of $\mathcal{L}$
\begin{align*}
\operatorname{div}_b\left(\mathcal{L}(B^{\eps})\nabla_b f(\cdot,B^\eps)\right)&=\sum_{\substack{k\in \Z^3_0\\ n\leq \lvert k\rvert \leq 2n}}\sum_{\alpha,\beta=1}^3\partial_{b,\alpha}\left(\frac{\left(B^\eps\cdot k\right)^2}{\lvert k\rvert^5}\left(I-\frac{k\otimes k}{\lvert k\rvert^2}\right)_{\alpha,\beta}\partial_{b,\beta }f(\cdot,B^{\eps})\right)\\ & =  2\sum_{\substack{k\in \Z^3_0\\ n\leq \lvert k\rvert \leq 2n}}   \frac{B^\eps\cdot k }{\lvert k\rvert^5} k\cdot \left(I-\frac{k\otimes k}{\lvert k\rvert^2}\right) \nabla_b f(\cdot,B^{\eps})) \\ &+\sum_{\substack{k\in \Z^3_0\\ n\leq \lvert k\rvert \leq 2n}}\sum_{\alpha,\beta=1}^3\frac{(B^{\eps}\cdot k)^2}{\lvert k\rvert^5} \left(I-\frac{k\otimes k}{\lvert k\rvert^2}\right)_{\alpha,\beta}\partial_{b,\alpha,\beta} f(\cdot, B^{\eps})\\ &=\sum_{\substack{k\in \Z^3_0\\ n\leq \lvert k\rvert \leq 2n}}\sum_{\alpha,\beta=1}^3\frac{(B^{\eps}\cdot k)^2}{\lvert k\rvert^5} \left(I-\frac{k\otimes k}{\lvert k\rvert^2}\right)_{\alpha,\beta}\partial_{b,\alpha,\beta} f(\cdot, B^{\eps}).
\end{align*}
Therefore
\begin{multline}\label{vlasov_pde_step_4}
\sumkj \int_0^t\int_{\T^3}Tr\left(\nabla^2_b f(x,B^{\eps}_s(x))(B^{\eps}_s(x)\cdot\nabla\sigma_{k,j}(x))\otimes (B^{\eps}_s(x)\cdot\nabla\sigma_{-k,j}(x))\right)dxds\\=\int_0^t \int_{\T^3} \operatorname{div}_b\left(\mathcal{L}(B^{\eps}(x))\nabla_b f(x,B^\eps(x))\right) dx ds. 
\end{multline}
Lastly we note that for each $\alpha,\beta\in \{1,2,3\}$
\begin{align*}
    \sumkj (\sigma_{k,j}\cdot\nabla B^{\eps})_{\alpha} (\sigma_{k,j}\cdot\nabla B^{\eps})_{\beta}&=\frac{2}{3}\eta\sum_{\gamma=1}^3 \partial_\gamma (B^{\eps})_{\alpha}\partial_\gamma (B^{\eps})_{\beta}.
\end{align*}
Therefore, by the chain rule and integrating by parts repeatedly
\begin{align}\label{vlasov_pde_step_5}
&\frac{2}{3}\eta\int_0^t \int_{\T^3}\nabla_b f(x,B^{\eps}_s(x))\cdot \Delta B^{\eps}_s(x)dxds\notag\\ &+\sumkj \int_0^t\int_{\T^3}Tr\left(\nabla^2_b f(x,B^{\eps}_s(x))(\sigma_{k,j}(x)\cdot\nabla B^{\eps}_s(x))\otimes (\sigma_{-k,j}(x)\cdot\nabla B^{\eps}_s(x))\right)dxds \notag \\ &= \frac{2}{3}\eta \sum_{\alpha,\beta=1}^3\int_0^t\int_{\T^3} \partial_{x,\alpha}(\partial_{b,\beta}f(x,B^{\eps}_s(x))\partial_{\alpha}(B^{\eps}_s(x))_{\beta})-(\partial_{x,\alpha}\partial_{b,\beta}f)(x,B^{\eps}_s(x))\partial_{\alpha}(B^{\eps}_s(x))_{\beta} dx ds\notag\\ & =   -\frac{2}{3}\eta \sum_{\alpha,\beta=1}^3\int_0^t\int_{\T^3} \partial_{x,\alpha}\left(\partial_{x,\alpha}f)(x,B^{\eps}_s(x))\right) dx ds+\frac{2}{3}\eta\int_0^t \int_{\T^3}(\Delta_x f)(x,B^{\eps}_s(x)) dx ds\notag\\ &= \frac{2}{3}\eta\int_0^t \int_{\T^3}(\Delta_x f)(x,B^{\eps}_s(x)) dx ds. 
\end{align}
Integrating \eqref{vlasov_pde_step_1} on $\T^3$, thanks to \eqref{vlasov_pde_step_2}, \eqref{vlasov_pde_step_3}, \eqref{vlasov_pde_step_4} and \eqref{vlasov_pde_step_5} the claim follows.
\end{proof}
Now we are in the position to prove the validity of relation \eqref{vlasov_weak}.
\begin{proof}[Proof of \eqref{vlasov_weak}]
Let us start considering $f\in C^{3}_b(\T^3\times \R^3)$ and $t\in [0,T]$. Since $B\in C([0,T];V_w)$ $ \mathbb{P}-a.s.$ then for each $t\in [0,T]$ $B^{\eps}(t)\rightarrow B(t)$ in $V\hookrightarrow L^p(\T^3;\R^3)$ for all $p\leq 6$ and 
\begin{align}\label{uniform_bound_regularization}
    \sup_{t\in [0,T]}\norm{B^{\eps}(t)}_V\leq \sup_{t\in [0,T]}\norm{B(t)}_V<+\infty \quad \mathbb{P}-a.s.
\end{align}
Due to \autoref{regularized vlasov weak} it is enough to pass to the limit in each term appearing in the lemma. Thanks to the regularity of $f$ we get easily
\begin{align*}
    \left \lvert \int_{\T^3}f(x,B^{\eps}_t(x))-f(x,B_t(x)) dx\right \rvert& \leq \norm{f}_{C^1_b}\norm{B^{\eps}_t-B_t}_{L^1(\T^3)}\rightarrow 0\quad \mathbb{P}-a.s.
\end{align*}
This gives us the convergence of the terms non integrated in time. Secondly let us study the deterministic integrals affecting the limit. Due to the regularity of $f$ and \eqref{uniform_bound_regularization} we have
\begin{align*}
\left \lvert \int_{\T^3}(\Delta_x f)(x,B^{\eps}_t(x))dx\right \rvert&\lesssim \norm{f}_{C^2_b}\in L^1(0,T)\quad \mathbb{P}-a.s.,\\
\left \lvert \int_{\T^3}\operatorname{div}_b\left(\mathcal{L}(B^{\eps}_t(x))\nabla_b f(x,B^\eps_t(x))\right) dx\right \rvert& \lesssim \norm{f}_{C^2_b}\norm{B}_{C([0,T];H)}^2\in L^1(0,T)\quad \mathbb{P}-a.s., \\
 \left \lvert \int_{\T^3}(\Delta_x f)(x,B^{\eps}_t(x))-(\Delta_x f)(x,B_t(x)) dx\right \rvert& \leq \norm{f}_{C^3_b}\norm{B^{\eps}_t-B_t}_{L^1(\T^3)}\rightarrow 0\quad \mathbb{P}-a.s., 
\end{align*}
\begin{multline*}
 \left \lvert \int_{\T^3}\operatorname{div}_b\left(\mathcal{L}(B^{\eps}_t(x))\nabla_b f(x,B^\eps_t(x))-\mathcal{L}(B_t(x))\nabla_b f(x,B_t(x))\right) dx\right \rvert\\ \leq 
 \left \lvert \int_{\T^3}\operatorname{div}_b\left(\mathcal{L}(B^{\eps}_t(x))\nabla_b f(x,B^\eps_t(x))-\mathcal{L}(B^{\eps}_t(x))\nabla_b f(x,B_t(x))\right) dx\right \rvert\\ +
 \left \lvert \int_{\T^3}\operatorname{div}_b\left(\mathcal{L}(B_t(x))\nabla_b f(x,B^\eps_t(x))-\mathcal{L}(B_t(x))\nabla_b f(x,B_t(x))\right) dx\right \rvert\\ \lesssim \norm{f}_{C^3_b}\norm{B^{\eps}_t}_{L^3(\T^3)}^2\norm{B^{\eps}_t-B_t}_{L^3(\T^3)}+\norm{f}_{C^2_b}\norm{B^{\eps}_t-B_t}_{L^2(\T^3)}\norm{B^{\eps}_t+B_t}_{L^2(\T^3)}\rightarrow 0\quad \mathbb{P}-a.s.
\end{multline*}
Therefore by dominated convergence theorem we get the convergence of the corresponding terms. The analysis of the stochastic integrals affecting the limit goes in a similar way. Indeed by It\^o isometry
\begin{align*}
&\expt{\left\lvert \sumkj\int_0^t\int_{\T^3}\left(B^\eps_s(x)\cdot\nabla\sigma_{k,j}(x)\nabla_b f(x,B^\eps_s(x))-B_s(x)\cdot\nabla\sigma_{k,j}(x)\nabla_b f(x,B_s(x))\right)dx dW^{k,j}_s\right\rvert^2}\\ &= \expt{ \int_0^t\sumkj\left\lvert\int_{\T^3}\left(B^\eps_s(x)\cdot\nabla\sigma_{k,j}(x)\nabla_b f(x,B^\eps_s(x))-B_s(x)\cdot\nabla\sigma_{k,j}(x)\nabla_b f(x,B_s(x))\right)dx\right\rvert^2 ds},\\
&\expt{\left\lvert \sumkj\int_0^t\int_{\T^3}\left(\sigma_{k,j}(x)\cdot(\nabla_x f)(x,B^\eps_s(x))-\sigma_{k,j}(x)\cdot(\nabla_x f)(x,B_s(x))\right)dx dW^{k,j}_s\right\rvert^2}\\ &=\expt{ \int_0^t\sumkj\left\lvert\int_{\T^3}\left(\sigma_{k,j}(x)\cdot(\nabla_x f)(x,B^\eps_s(x))-\sigma_{k,j}(x)\cdot(\nabla_x f)(x,B_s(x))\right)dx\right\rvert^2 ds}.
\end{align*}
Again due to \autoref{well_posedness_estimates_MF}, \eqref{uniform_bound_regularization} and the regularity of $f$
\begin{align*}
&\sumkj\left\lvert\int_{\T^3}\left(B^\eps_t(x)\cdot\nabla\sigma_{k,j}(x)\nabla_b f(x,B^\eps_t(x))-B_t(x)\cdot\nabla\sigma_{k,j}(x)\nabla_b f(x,B_t(x))\right)dx\right\rvert^2\\ & \lesssim   \norm{f}_{C^1_b}^2 \norm{B}_{C([0,T];H)}^2\in L^1(\Omega\times (0,T)),\\
&\sumkj\left\lvert\int_{\T^3}\left(\sigma_{k,j}(x)\cdot(\nabla_x f)(x,B^\eps_t(x))-\sigma_{k,j}(x)\cdot(\nabla_x f)(x,B_t(x))\right)dx\right\rvert^2 \lesssim \norm{f}_{C^1_b}^2 \in L^1(\Omega\times (0,T))
\end{align*}
and
\begin{align*}
&\sumkj\left\lvert\int_{\T^3}\left(\sigma_{k,j}(x)\cdot(\nabla_x f)(x,B^\eps_t(x))-\sigma_{k,j}(x)\cdot(\nabla_x f)(x,B_t(x))\right)dx\right\rvert^2\\ & \lesssim \norm{f}^2_{C^2_b}\norm{B^\eps_t-B_t}^2\rightarrow 0\quad a.e.\quad (\omega,t)\in \Omega\times (0,T)    
\end{align*}
\begin{multline*}
\sumkj\left\lvert\int_{\T^3}\left(B^\eps_t(x)\cdot\nabla\sigma_{k,j}(x)\nabla_b f(x,B^\eps_t(x))-B_t(x)\cdot\nabla\sigma_{k,j}(x)\nabla_b f(x,B_t(x))\right)dx\right\rvert^2\\  \lesssim    \sumkj\left\lvert\int_{\T^3}\left(B^\eps_t(x)\cdot\nabla\sigma_{k,j}(x)\nabla_b f(x,B^\eps_t(x))-B^{\eps}_t(x)\cdot\nabla\sigma_{k,j}(x)\nabla_b f(x,B_t(x))\right)dx\right\rvert^2\\ +\sumkj\left\lvert\int_{\T^3}\left(B^\eps_t(x)\cdot\nabla\sigma_{k,j}(x)\nabla_b f(x,B(x))-B_t(x)\cdot\nabla\sigma_{k,j}(x)\nabla_b f(x,B_t(x))\right)dx\right\rvert^2\\ \lesssim \norm{f}^2_{C^2_b}\norm{B_t}_{L^2}^2 \norm{B^{\eps}_t-B_t}_{L^2}^2+\norm{f}_{C^1_b}^2 \norm{B^{\eps}_t-B_t}_{L^1}^2\rightarrow 0\quad a.e.\quad (\omega,t)\in \Omega\times (0,T).
\end{multline*}
Therefore by dominated convergence theorem we get the convergence of the corresponding terms. Up to passing to subsequences the convergence is $\mathbb{P}-a.s.$ We are left to show that the remaining terms $\mathbb{P}-a.s.$ converge to $0$. We start with the deterministic integrals. Thanks to \eqref{uniform_bound_regularization} and the regularity of $f$
\begin{multline*}
    \sumkj \left\lvert\int_{\T^3}Tr\left(\nabla^2_b f(x,B^{\eps}_t(x))\left(\left([\sigma_{k,j}\cdot\nabla, \varrho_{\eps}]+[\cdot\nabla \sigma_{k,j}, \varrho_{\eps} ]\right)B_t(x)\right.\right.\right.  \\ \otimes \left.\left.\left.\left([\sigma_{-k,j}\cdot\nabla, \varrho_{\eps}]+[\cdot\nabla \sigma_{-k,j}, \varrho_{\eps} ]\right)B_t(x)\right)\right) dx \right\rvert\\ \lesssim \norm{f}_{C^2_b}\sumkj\left( \norm{[\sigma_{k,j}\cdot\nabla,\varrho_{\eps}]B}_{L^2(\T^3)}^2+\norm{[\cdot\nabla \sigma_{k,j},\varrho_{\eps}]B}_{L^2(\T^3)}^2\right),
\end{multline*}
\begin{multline*}
    \sumkj\left\lvert \int_{\T^3}Tr\left(\nabla^2_b f(x,B^{\eps}_t(x))\left(\left([\sigma_{k,j}\cdot\nabla, \varrho_{\eps}]+[\cdot\nabla \sigma_{k,j}, \varrho_{\eps} ]\right)B_t(x)\right.\right. \right. \\   \left.\left. \otimes\left(\sigma_{-k,j}(x)\cdot \nabla B^{\eps}_t(x)-B^{\eps}_t(x)\cdot\nabla \sigma_{-k,j}(x)\right)\right)\right) dx \bigg\rvert\\ \lesssim \norm{f}_{C^2_b}\sup_{t\in [0,T]}\norm{B_t}_{V}\sumkj\left( \norm{[\sigma_{k,j}\cdot\nabla,\varrho_{\eps}]B}_{L^2(\T^3)}+\norm{[\cdot\nabla \sigma_{k,j},\varrho_{\eps}]B}_{L^2(\T^3)}\right).
\end{multline*}
As a consequence of \autoref{well_posedness_estimates_MF}, \autoref{commutator_lemma} the terms above converge pointwise to $0$ and are uniformly bounded by a function in $L^1(0,T)$. Therefore dominated convergence theorem implies the claim. Concerning the last stochastic integral, by It\^o isometry
\begin{align*}
&\expt{\left\lvert\sumkj \int_0^t \int_{\T^3}\left([\sigma_{k,j}\cdot\nabla, \varrho_{\eps}]+[\cdot\nabla \sigma_{k,j}, \varrho_{\eps} ]\right)B_s(x)\cdot\nabla_b f(x,B^{\eps}_s) dx dW^{k,j}_s \right\rvert^2}\\ & =  \expt{ \int_0^t \sumkj \left\lvert\int_{\T^3}\left([\sigma_{k,j}\cdot\nabla, \varrho_{\eps}]+[\cdot\nabla \sigma_{k,j}, \varrho_{\eps} ]\right)B_s(x)\cdot\nabla_b f(x,B^{\eps}_s) dx \right\rvert^2 ds }.  
\end{align*}
If we show that the latter term converges to $0$ as $\eps\rightarrow 0$ then by passing to subsequences we have the required $\mathbb{P}-a.s.$ convergence. Thanks to the regularity of $f$ 
\begin{multline*}
    \sumkj \left\lvert\int_{\T^3}\left([\sigma_{k,j}\cdot\nabla, \varrho_{\eps}]+[\cdot\nabla \sigma_{k,j}, \varrho_{\eps} ]\right)B_t(x)\cdot\nabla_b f(x,B^{\eps}_t) dx \right\rvert^2\\ \lesssim \norm{f}_{C^1_b}^2\sumkj\left( \norm{[\sigma_{k,j}\cdot\nabla,\varrho_{\eps}]B}_{L^2(\T^3)}^2+\norm{[\cdot\nabla \sigma_{k,j},\varrho_{\eps}]B}_{L^2(\T^3)}^2\right).
\end{multline*}
As a consequence of \autoref{well_posedness_estimates_MF}, \autoref{commutator_lemma} the terms above converge pointwise to $0$ and are uniformly bounded by a function in $L^1(\Omega\times(0,T))$. Dominated convergence theorem's implies the validity of the claim. \\
So far we showed that for each $t\in [0,T]$ and $f\in C^3_b(\T^3\times \R^3)$ there exists a zero measure set $\mathcal{N}\subseteq \Omega$ such that on its complementary
\begin{align*}
    \int_{\T^3}f(x,B_t(x))-f(x,\overline{B}_0(x)) dx&=-\sumkj\int_0^t\int_{\T^3} \sigma_{k,j}(x)\cdot\left(\nabla_x f\right)(x,B_s(x))dW^{k,j}_s\\ & -\sumkj\int_0^t\int_{\T^3}B_s(x)\cdot\nabla\sigma_{k,j}(x)\nabla_b f(x,B_s(x))dW^{k,j}_s \notag\\ & +\frac{2}{3}\eta\int_0^t\int_{\T^3} \left(\Delta_x f\right)(x,B_s(x)) dxds\\ & +\int_0^t\int_{\T^3} \operatorname{div}_b\left(\mathcal{L}(B_s(x))\nabla_b f(x,B_s(x))\right) dx ds.
\end{align*}
Thanks to the density of $C^{3}_b(\T^3\times \R^3)$ in $C^{2}_b(\T^3\times \R^3)$ and the continuity properties of $B$, \emph{cf. }\autoref{well_posedness_estimates_MF}, the relation above holds for each $\phi \in C^2_b(\T^3\times \R^3)$ $\mathbb{P}-a.s.$  uniformly in $t\in [0,T]$. Recalling the definition of $\mu(t,x,db)$ this completes the proof of \eqref{vlasov_weak}.
\end{proof}
\subsection*{Uniform Estimates}
Now we are interested to study the regularity of the increments of $\int_{T^3\times \R^3} f(x,b)\mu_n(t,x,db)dx$. The following holds
\begin{proposition}\label{a_priori_time}
For each $f\in C^2_b(\T^3\times\R^3)$ and $\gamma\geq 1$ we have
\begin{align*}
    \expt{\left\lvert \int_{T^3\times \R^3} f(x,b)(\mu_n(t,x,db)-\mu_n(s,x,db))dx \right\rvert^{\gamma}}\leq  C_2(\overline{B}_0,T,2\gamma)\norm{f}_{C^2_b(\T^3\times\R^3)}^{\gamma}\lvert t-s\rvert^{\gamma/2}.
\end{align*}
\end{proposition}
\begin{proof}
Due to \autoref{alter_vlasov_weak} we have
\begin{align*}
\expt{\left\lvert \int_{\T^3}f(x,B^n_t(x))-f(x,B^n_s(x)) dx\right\rvert^{\gamma}} &\lesssim_{\gamma}\expt{\left\lvert \sumkj\int_s^t\int_{\T^3} \skj(x)\cdot\left(\nabla_x f\right)(x,B^n_r(x))dxdW^{k,j}_r \right\rvert^\gamma}\\ &+\expt{\left\lvert \sumkj\int_s^t\int_{\T^3}B^n_r(x)\cdot\nabla\skj(x)\nabla_b f(x,B^n_r(x))dxdW^{k,j}_r\right\rvert^\gamma}\\ & +\left(\frac{2}{3}\eta_n\right)^{\gamma}\expt{\left\lvert \int_s^t\int_{\T^3} \left(\Delta_x f\right)(x,B^n_r(x)) dxdr\right\rvert^\gamma} \\ & +\expt{\left\lvert \int_s^t\int_{\T^3} \operatorname{div}_b\left(\mathcal{L}^n(B^n_r(x))\nabla_b f(x,B^n_r(x))\right) dx dr\right\rvert^\gamma}\\ & =J_1+J_2+J_3+J_4.   \end{align*}
We have immediately
\begin{align}\label{estimate_J_3}
    J_3\lesssim \norm{\nabla^2_x f}_{C_b}^{\gamma}\lvert t-s\rvert^{\gamma}.
\end{align}
Similarly we can treat $J_4$. Indeed, since $k\cdot \left(I-\frac{k\otimes k}{\lvert k\rvert^2}\right)=0$, we obtain   
\begin{align}\label{estimate J_4}
    J_4 & \lesssim_{\gamma}\expt{\left\lvert \int_s^t\int_{\T^3} Tr\left(\mathcal{L}^n(B^n_r(x))\nabla^2_b f(x,B^n_r(x))\right) dx dr\right\rvert^\gamma}\notag\\ & \lesssim 
    \expt{\left\lvert \int_s^t\int_{\T^3} \lvert B^n_r(x)\rvert^2 \sum_{\substack{k\in \Z^3_0\\ n\leq \lvert k\rvert \leq 2n}}\frac{1}{\lvert k\rvert^3}\norm{\nabla^2_b f}_{C_b} dx dr\right\rvert^\gamma}\notag\\ & \lesssim \alpha_n^{\gamma}\vert   t-s\vert^\gamma\norm{\nabla^2_b f}_{C_b}^{\gamma}\expt{\sup_{t\in [0,T]}\norm{B^n_t}^{2\gamma}}\notag\\ & \lesssim C_1(\overline{B}_0,T,2\gamma)\norm{\nabla^2_b f}_{C_b}^{\gamma}\lvert t-s\rvert^{\gamma}.
\end{align}
Lastly $J_1$ and $J_2$ can be treated by Burkholder-Davis-Gundy inequality. Indeed,
\begin{align}\label{estimate_J_1}
    J_1&\lesssim_{\gamma}\expt{\left(\sumkj \int_s^t \left(\int_{\T^3} \skj(x)\cdot\left(\nabla_x f\right)(x,B^n_r(x))dx\right)^2 dr\right)^{\gamma/2}}\notag\\ & \lesssim_{\gamma} \eta_n^{\gamma/2}\norm{\nabla_x f}_{C_b}^\gamma\lvert t-s\rvert^{\gamma/2}.
\end{align}
Finally
\begin{align}\label{estimate_J_2}
J_2 &\lesssim_{\gamma} \expt{\left(\sumkj \int_s^t \left(\int_{\T^3} B^n_r(x)\cdot\nabla\skj(x)\nabla_b f(x,B^n_r(x))dx\right)^2 dr\right)^{\gamma/2}} \notag \\ & \lesssim \alpha_n^{\gamma/2}\expt{\operatorname{sup}_{t\in [0,T]}\norm{B^n_t}^{\gamma}} \norm{\nabla_b f}_{C_b}^\gamma \lvert t-s\rvert^{\gamma/2}\notag\\ & \lesssim\alpha_n^{\gamma/2}C_{1}(\overline{B}_0,T,\gamma)\norm{\nabla_b f}_{C_b}^\gamma \lvert t-s\rvert^{\gamma/2}.
\end{align}
Combining \eqref{estimate_J_1}, \eqref{estimate_J_2}, \eqref{estimate_J_3}, \eqref{estimate J_4} the result follows.
\end{proof}
As an immediate corollary of \autoref{a_priori_time} and of the Garsia--Rodemich--Rumsey inequality, see \cite{garsia1970real}, we obtain
\begin{corollary}\label{cor_holder_norms}
For each $f\in C^2_b(\T^3\times\R^3),\ \phi,\lambda$ such that $\phi<\frac{1}{2},\ \phi-\frac{1}{\lambda}>0$ then \begin{align*}
    \int_{T^3\times \R^3}f(x,b)\mu_n(t,x,db)dx
\end{align*} has a $\phi-\frac{1}{\lambda}$ H\"older version. Moreover
\begin{align*}
    \expt{\norm{\int_{T^3\times \R^3}f(x,b)\mu_n(t,x,db)dx}_{C^{0,\phi-\frac{1}{\lambda}}}^{\lambda}}\leq C(\phi,\lambda,\overline{B}_0,T)\norm{f}_{C^2_b}^{\lambda}.
\end{align*}
\end{corollary}
\begin{proof}
Thanks to Garsia--Rodemich--Rumsey inequality we know that for each $t,s$ \begin{multline}\label{grr ineq}
    \frac{\left\lvert\int_{T^3\times \R^3}f(x,b)\left(\mu_n(t,x,db)-\mu_n(s,x,db)\right)dx\right\rvert^{\lambda}}{\lvert t-s\rvert^{\phi\lambda-1}}\\ \leq C(\phi,\lambda)\int_s^t\int_s^t \frac{\left\lvert\int_{T^3\times \R^3}f(x,b)\left(\mu_n(r,x,db)-\mu_n(v,x,db)\right)dx\right\rvert^{\lambda}}{\lvert r-v\rvert^{\phi\lambda+1}}dr dv.
\end{multline}    
Therefore, combining \eqref{grr ineq} and  \autoref{a_priori_time} we obtain
\begin{align*}
\expt{\norm{\int_{T^3\times \R^3}f(x,b)\mu_n(t,x,db)dx}_{C^{0,\phi-\frac{1}{\lambda}}}^{\lambda}}&\lesssim_{\lambda} \norm{f}_{C_b}^{\lambda}+\int_0^T\int_0^T \frac{C_2(\overline{B}_0,T,2\lambda)\norm{f}_{C^2_b(\T^3\times\R^3)}^{\lambda}}{\lvert r-v\rvert^{\lambda(\phi-\frac{1}{2})+1}} dr dv \\ & \leq C(\phi,\lambda,\overline{B}_0,T)\norm{f}_{C^2_b}^{\lambda}.    
\end{align*}
This completes the proof.
\end{proof}
\subsection{Compactness of the laws}\label{compactness_stochastic_young_measures}
The goal of this section is to show that the laws of $\mu_n$ are tight in the space of probability measure on $C([0,T];\mathcal{Y}).$ We start introducing a family of mollifiers on $\T^3\times \R^3$, i.e. we introduce \begin{align*}
    h:\T^3\times \R^3\rightarrow \R\quad s.t.\ \operatorname{supp}h\subseteq B(\mathbf{0},1/2),\ h\in C^{\infty}(\T^3\times\R^3),\ h\geq 0, \int_{\T^3\times\R^3}h(\mathbf{x})d\mathbf{x}=1
\end{align*}
and set
$h_\eps:\T^3\times \R^3\rightarrow \R,\quad h_{\eps}(\mathbf{x})=\frac{1}{\eps^6}h(\frac{\mathbf{x}}{\eps})$
We recall that if $f\in C_{b}(\T^3\times\R^3)$ is uniformly continuous then
\begin{align}\label{properties mollification}
f^{\eps}:=h_{\eps}\ast f\rightarrow f\quad \text{in } C_{b}(\T^3\times\R^3),\quad 
\norm{f^{\eps}}_{C^1_b}\leq \frac{\norm{f}_{C_b}\norm{\nabla h}_{W^{1,1}}}{\eps},\quad \norm{f^{\eps}}_{C^2_b}\leq \frac{\norm{f}_{C_b}\norm{ h}_{W^{2,1}}}{\eps^2}.  
\end{align}
Now we are ready to prove the main result of this section.
\begin{theorem}\label{compactness laws}
The laws of $\mu_n$ are tight in the space of probability measure on $C([0,T];\mathcal{Y}).$    
\end{theorem}
\begin{proof}
Thanks to \autoref{compact_sets} it is enough to show that for each $\delta>0$ there exist a family of positive quantities $R(k,\delta),\ \theta(k,\delta,i),\ i,k\in \N$ such that
\begin{align*}
    \mathbb{P}(\mu_n\notin \overline{K^{\delta}})\leq \delta.
\end{align*}
By definition of $K^{\delta}$ we have thanks to Markov's inequality
\begin{align*}
    \mathbb{P}(\mu_n\notin \overline{K^{\delta}})& \leq  \mathbb{P}(\mu_n\notin {K^{\delta}})\\ & \leq \sum_{k\in \N}\mathbb{P}\left(\operatorname{sup}_{t\in [0,T]}\int_{\T^3}^* \mu_n(t, x, B(0,R(k,\delta)))dx> \frac{1}{k}\right)\\ & +\sum_{i, k\in \N}\mathbb{P}\left(\sup_{\lvert t-s\rvert< \theta(k,\delta,i)}\left\lvert \int_{\T^3\times \R^3} g_i(x,b) (\mu_n(t,x,db)-\mu_n(s,x,db))\right\rvert > \frac{1}{k}\right)\\ & \leq \sum_{k\in \N} k\expt{\operatorname{sup}_{t\in [0,T]}\int_{\T^3}^* \mu_n(t,x, B(0,R(k,\delta)))dx}\\ & + \sum_{i, k\in \N}k\expt{\sup_{\lvert t-s\rvert< \theta(k,\delta,i)}\left\lvert \int_{\T^3\times \R^3} g_i(x,b) (\mu_n(t,x,db)-\mu_n(s,x,db))\right\rvert}=H_1+H_2.
\end{align*}
Thanks to the definition of $\mu_n$ and \autoref{well_posedness_estimates_MF} we can easily treat $H_1$ again by Markov's inequality. Indeed, 
\begin{align}\label{estimate_1_compact_meas}
    H_1&\leq \sum_{k\in \N} k\expt{\operatorname{sup}_{t\in [0,T]}\int_{\T^3}\one_{|B^n_t(x)|>R(k,\delta)}(x) dx }\notag\\ & \leq \lvert \T^3\rvert \sum_{k\in \N} \frac{k}{R^2(k,\delta)}\expt{\operatorname{sup}_{t\in [0,T]}\norm{B^n_t}^2}\notag\\ & \leq C_0(\overline{B}_0,T)\lvert \T^3\rvert \sum_{k\in \N} \frac{k}{R^2(k,\delta)}.
\end{align}
The estimate of $H_2$ is more involved. Obviously we have 
\begin{align*}
    \left\lvert \int_{\T^3\times \R^3} g_i(x,b) (\mu_n(t,x,db)-\mu_n(s,x,db))\right\rvert&\leq \left\lvert \int_{\T^3\times \R^3} g^{\eps}_i(x,b) (\mu_n(t,x,db)-\mu_n(s,x,db))\right\rvert\\ &+ \int_{\T^3\times \R^3}\lvert  g_i(x,b)-g^{\eps}_i(x,b)\rvert \lvert\mu_n(t,x,db)-\mu_n(s,x,db)\rvert\\ & \leq \lvert \T^3\rvert \norm{g_i-g_i^{\eps}}_{C_b}\\ &+\left\lvert \int_{\T^3\times \R^3} g^{\eps}_i(x,b) (\mu_n(t,x,db)-\mu_n(s,x,db))\right\rvert.
\end{align*}
Due to \eqref{properties mollification}  for each $i\in \N$ we can choose $\eps=\eps(k,\delta,i)$ such that 
\begin{align*}
    \lvert \T^3\rvert \norm{g_i-g_i^{\eps}}_{C_b}\leq \frac{3\delta}{2^{i+1} \pi^2 k^3}.
\end{align*}
Therefore we obtain
\begin{align*}
    H_2\leq \frac{\delta}{4}+\sum_{i, k\in \N}k\expt{\sup_{\lvert t-s\rvert< \theta(k,\delta,i)}\left\lvert \int_{\T^3\times \R^3} g^{\eps}_i(x,b) (\mu_n(t,x,db)-\mu_n(s,x,db))\right\rvert}.
\end{align*}
Lastly due to \autoref{cor_holder_norms}, with the choice of $\phi=\frac{3}{8},\ \gamma=8$, and \eqref{properties mollification} we have
\begin{align}\label{estimate_2_compact_meas}
    H_2\leq \frac{\delta}{4}+\frac{C(\overline{B}_0,T)\norm{h}_{W^{2,1}}}{\eps^2}\sum_{i, k\in \N}k\left(\theta(k,\delta,i)\right)^{1/4}.
\end{align}
Combining \eqref{estimate_1_compact_meas} and \eqref{estimate_2_compact_meas} we obtain
\begin{align*}
    \mathbb{P}(\mu_n\notin \overline{K^{\delta}})& \leq \frac{\delta}{4}+C_0(\overline{B}_0,T)\lvert \T^3\rvert \sum_{k\in \N} \frac{k}{R^2(k,\delta)}+\frac{C_3(\overline{B}_0,T)\norm{h}_{W^{2,1}}}{\eps^2}\sum_{i, k\in \N}k\left(\theta(k,\delta,i)\right)^{1/4}.
\end{align*}
Choosing 
\begin{align*}
    R(k,\delta)=\sqrt{\frac{2\pi^3k^3 C_0(\overline{B}_0,T)|\T^3|}{3\delta}},\quad \theta(k,\delta,i)=\left(\frac{3\eps^2}{2^{i+1}\pi^2k^3\norm{h}_{W^{2,1}}C_3(\overline{B}_0,T)}\right)^4
\end{align*}
the claim follows and the proof is complete.
\end{proof}
\subsection{Passage to the Limit}\label{passage_to_the_limit_sec}
Thanks to \autoref{compactness laws}, by Jakubowski version of Skorokhod's representation's theorem, \cite[Theorem 2]{yakubovskii1997almost}, up to passing to subsequences, we can find an auxiliary probability space, that for simplicity we continue to call $(\Omega,\mathcal{F},\mathbb{P})$, and processes \begin{align*}
    \left(\Tilde{\mu}_n,\ W^n:=\{W^{n,k,j}\}_{\substack{k\in \mathbb{Z}^3_0\\ j\in \{1,2\}}}\right),\quad   \left(\overline{\mu},\ W:=\{W^{k,j}\}_{\substack{k\in \mathbb{Z}^3_0\\ j\in \{1,2\}}}\right),
\end{align*} such that \begin{align*}
    &\Tilde{\mu}_n\rightarrow \overline{\mu} \quad \text{in } C([0,T];\mathcal{Y}) \quad \mathbb{P}-a.s.\\
   & W^n\rightarrow W \quad \text{in } C([0,T];\mathbb{R}^{\mathbb{Z}^3_0\times \mathbb{Z}^3_0} )\quad \mathbb{P}-a.s.
\end{align*}
Of course the convergence above between $W^n$ and $W$ can be seen as the uniform convergence of cylindrical Wiener processes $W^n=\sumkj e_{k,j} W^{n,k,j},\ W=\sumkj e_{k,j} W^{k,j}$ on a suitable Hilbert space $U_0$.
 Before going on, in order to identify the PDE satisfied by $\overline{\mu}$ we provide integrability properties of $\overline{\mu}$. 
 \begin{proposition}\label{Limi_preserves_moments}
 The limit measure $\overline{\mu}$ satisfies
 \begin{align*}
\expt{ \sup_{t\in [0,T]}\int_{\T^3\times \R^3} \lvert b\rvert^2 \overline{\mu}(t,x,db)dx}\leq C(\overline{B}_0,T).    
\end{align*}
 \end{proposition}
\begin{proof}
Since $\Tilde{\mu}_n$ converges to $\overline{\mu}$ in $C([0,T];\mathcal{Y})$, for each $f\in C_{b}(\T^3\times \R^3)$ and $p\geq 1$, by dominated convergence theorem, 
\begin{align}\label{convergece_L^p_measures}
    \int_0^T\left\lvert \int_{\T^3\times \R^3} f(b,x) \tilde{\mu}_n(t,x,db)dx-\int_{\T^3\times \R^3} f(b,x) \overline{\mu}(t,x,db)dx\right\rvert^p  dt\rightarrow 0 \quad \mathbb{P}-a.s.
\end{align}
Choose a smooth test function such that
\begin{align*}
    \phi^M(b,x)=\begin{cases}
        \lvert b\rvert^2\quad &\text{if } \lvert b\rvert\leq  M\\
        0\quad &\text{if } \lvert b\rvert\geq  M+1,
    \end{cases}\quad \norm{\phi^M}_{C_b}\leq M^2. 
\end{align*}
Obviously $\phi^M(b,x)\nearrow \lvert b\rvert^2$ pointwise. Therefore, due to \eqref{convergece_L^p_measures}, Fatou's Lemma and \autoref{well_posedness_estimates_MF}
\begin{align*}
    \expt{ \left(\int_0^T\left\lvert \int_{\T^3\times \R^3} \phi^M(b,x) \overline{\mu}(t,x,db)dx\right\rvert^p  dt\right)^{1/p}}&\leq \liminf_{n\rightarrow +\infty} \expt{ \left(\int_0^T\left\lvert \int_{\T^3\times \R^3} \phi^M(b,x) \tilde{\mu}_n(t,x,db)dx\right\rvert^p  dt\right)^{1/p}}\\ & \leq \liminf_{n\rightarrow +\infty} \expt{ \operatorname{sup}_{t\in [0,T]} \int_{\T^3\times \R^3} \lvert b\rvert^2 \tilde{\mu}_n(t,x,db)dx}\\ & \leq C_0(\overline{B}_0,T).
\end{align*}
Then, by monotone convergence theorem, for each $p\geq 1$
\begin{align}\label{bound_l^p_measure}
    \expt{ \left(\int_0^T\left\lvert \int_{\T^3\times \R^3} \lvert b\rvert^2 \overline{\mu}(t,x,db)dx\right\rvert^p  dt\right)^{1/p}}& \leq C_0(\overline{B}_0,T).
\end{align}
This implies that $\norm{\int_{\T^3\times \R^3} \lvert b\rvert^2 \overline{\mu}(\cdot, x,db)dx}_{L^p(0,T)}<+\infty\quad \mathbb{P}-a.s.$ Therefore
\begin{align}\label{convergence_linftymeasure}
  \norm{\int_{\T^3\times \R^3} \lvert b\rvert^2 \overline{\mu}(\cdot,x,db)dx}_{L^p(0,T)}\rightarrow \norm{\int_{\T^3\times \R^3} \lvert b\rvert^2 \overline{\mu}(\cdot, x,db)dx}_{L^\infty(0,T)}.  
\end{align}
Combining \eqref{bound_l^p_measure} and \eqref{convergence_linftymeasure} we obtain the result by Fatou's lemma and the continuity of $\overline{\mu}$.
\end{proof}
By the properties of Riemann sums we have
\begin{align}\label{convergence_operators}
    &\mathcal{L}^n(b)=(b\otimes b):\left(\sum_{\substack{k\in \Z^3_0\\ n\leq \lvert k\rvert \leq 2n}}\frac{k\otimes k}{\lvert k\rvert^5}\otimes \left(I-\frac{k\otimes k}{\lvert k\rvert^2}\right)\right)\notag\\ & \rightarrow \mathcal{L}(b):=(b\otimes b):\int_{\xi\in \R^3,\ 1\leq \lvert \xi\rvert\leq 2} \frac{\xi\otimes\xi}{\lvert \xi\rvert^5}\otimes \left(I-\frac{\xi\otimes \xi}{\lvert \xi\rvert^2}\right)d\xi =\frac{8\pi \log 2}{15}\left(2\lvert b\rvert^2 I-b\otimes b\right)
\end{align}
uniformly on compact sets.\\
Let us introduce our limit pde on the space of Young measures
\begin{align}\label{limit_PDE_vlasov}
\begin{cases}
\partial_t \mu(t,x,db)&=\operatorname{div}(\mathcal{L}(b)\mu(t,x,db))\quad b\in \R^3,x\in \T^3\ t\in [0,T]\\
\mu(0,x,db)&=\delta_{\overline{B}_0(x)}(db)
\end{cases}
\end{align}
and its notion of solution
\begin{definition}\label{weak_vlasov_pde}
We say that $\mu\in C([0,T];\mathcal{Y})$ is a weak solution of equation \eqref{limit_PDE_vlasov} if 
\begin{align*}
\mu(0,x,db)=\delta_{\overline{B}_0(x)}(db),\quad \int_0^T \int_{\T^3\times \R^3}\lvert b\rvert^2 \mu(t,x,db)dx<+\infty
\end{align*}
and for each $f\in C_c^1((0,T);C^2_c(\T^3\times \R^3))$ it holds
\begin{align*}
  &\int_0^T\int_{\T^3\times\R^3}\partial_sf(s,x,b){\mu}(s,x,db)dxds+\int_0^T\int_{\T^3\times \R^3} \operatorname{div}_b\left(\mathcal{L}(b)\nabla_b f_s(x,b)\right){\mu}(s,x,db) dx ds=0.
  \end{align*}
\end{definition}
   \begin{remark}\label{rmk:indifference to time}
   Arguing as in \cite[Remark 2.3]{trevisan2016well}, if $\mu$ is a weak solution of \eqref{limit_PDE_vlasov} in the sense of \autoref{weak_vlasov_pde}, then
\begin{align*}
 &\int_{\T^3\times \R^3} f_{t_2}(x,b)\mu(t_2,x,db)dx- \int_{\T^3\times \R^3} f_{t_1}(x,b)\mu(t_1,x,db)dx \\&=\int_{t_1}^{t_2}\int_{\T^3\times\R^3}\partial_sf(s,x,b){\mu}(s,x,db)dxds+\int_{t_1}^{t_2}\int_{\T^3\times \R^3} \operatorname{div}_b\left(\mathcal{L}(b)\nabla_b f_s(x,b)\right)\mu(s,x,db) dx ds
  \end{align*}
for each 
   $0\leq t_1\leq t_2\leq T$ and $f\in C^1_b\left([0,T];C^2_b(\T^3\times\R^3)\right)$.
    \end{remark}
The uniqueness of solutions of \eqref{limit_PDE_vlasov} in the sense of \autoref{weak_vlasov_pde} is an easy application of the results of \cite{trevisan2016well}. Indeed, the following holds.
   \begin{proposition}\label{prop:uniqueness_vlasov}
       There exists a unique $\mu\in C([0,T];\mathcal{Y})$ which solves \eqref{limit_PDE_vlasov} in the sense of \autoref{weak_vlasov_pde}.
       \end{proposition}
   \begin{proof}
   We start observing that for each $b\in \R^3$
   \begin{align*}
   \mathcal{L}(b)=\frac{8\pi\log 2 \lvert b\rvert^2}{15} \frac{b\otimes b}{\lvert b\rvert^2}+\frac{16\pi\log 2 \lvert b\rvert^2}{15} \left(I-\frac{b\otimes b }{\lvert b\rvert^2}\right).
   \end{align*}
   Therefore $ \mathcal{L}(b)=\mathcal{A}(b)\mathcal{A}(b)$, where 
   \begin{align*}
   \mathcal{A}(b)=\frac{\sqrt{8\pi\log 2} \lvert b\rvert}{\sqrt{15}} \frac{b\otimes b}{\lvert b\rvert^2}+\frac{4\sqrt{\pi\log 2 }\lvert b\rvert}{\sqrt{15}} \left(I-\frac{b\otimes b }{\lvert b\rvert^2}\right).
   \end{align*}
In particular $\mathcal{A}(b)$ is a symmetric, non negative matrix for each $b\in \R^3$ and the function $b\rightarrow  \mathcal{A}(b)$ is Lipschitz with linear growth.
Since for each $j\in \{1,2,3\},\ \sum_{i=1}^3 \partial_i\mathcal{L}_{i,j}(b)=0$, if $\hat{W}$ is a 3D Brownian motion, then the stochastic differential equation \begin{align}\label{uniqueness_Sde}
\begin{cases}
dX_t&=0,\\
db_t&=\mathcal{A}(b_t)d\hat{W}_t\\
X_0&=x\\
b_0&=b
\end{cases}
\end{align}
has \eqref{limit_PDE_vlasov} as Fokker-Planck equation associated. We have uniqueness in law of weak solutions of the stochastic differential equation \eqref{uniqueness_Sde} due to the properties of $\mathcal{A}(b)$. The latter implies uniqueness of solutions of \eqref{limit_PDE_vlasov} due to \cite[Theorem 2.5]{trevisan2016well}.
   \end{proof}
   Therefore, we are left to show that our limit object $\overline{\mu}$ is a weak solution of \eqref{limit_PDE_vlasov}. This is the content of the following result. \begin{proposition}\label{prop_limit_PDE}
   $\overline{\mu}$ is the unique weak solution of \eqref{limit_PDE_vlasov} in the sense of \autoref{weak_vlasov_pde}. 
   \end{proposition}
   \begin{proof}
   We divide the proof in two steps.\\
   \emph{Step 1: test function independent of time}. Let us start showing that for each $f\in C^2_c(\T^3\times \R^3)$ and $t\in [0,T]$ \begin{multline}\label{limit_pde_weak_indep_time}
\int_{\T^3\times \R^3}f(x,b)(\overline{\mu}(t,x,db)-\delta_{\overline{B}_0(x)}(db)) dx\\ =
\int_0^t\int_{\T^3\times \R^3} \operatorname{div}_b\left(\mathcal{L}(b)\nabla_b f(x,b)\right)\overline{\mu}(s,x,db) dx ds  \quad \mathbb{P}-a.s.
   \end{multline}
By standard arguments, see for example \cite[Chapter 2]{flandoli2023stochastic}, also $\tilde{\mu}_n$ satisfies \eqref{vlasov_weak}.
    Now we pass to the limit in each term of \eqref{vlasov_weak}. Thanks to the a.s.  convergence of $\tilde{\mu}_n$ to $\overline{\mu}$ in $C([0,T];\mathcal{Y})$ we obtain
    \begin{align*}
    \int_{\T^3\times \R^3}f(x,b)(\tilde{\mu}_n(t,x,db)-\delta_{\overline{B}_0(x)}(db)) dx\rightarrow \int_{\T^3\times \R^3}f(x,b)(\overline{\mu}(t,x,db)-\delta_{\overline{B}_0(x)}(db)) dx\quad \mathbb{P}-a.s.    
    \end{align*}
    Secondly, as discussed at the beginning of the proof of \autoref{Limi_preserves_moments}, since $\eta
_n\rightarrow 0$, $\mathcal{L}^n(b)~\rightarrow~\mathcal{L}(b)$ uniformly on compact sets and $f$ has compact support we have
    \begin{align*}
    \left\lvert\eta_n\int_0^t\int_{\T^3\times \R^3} \Delta_x f(x,b)\tilde{\mu}_n(s,x,db) dx ds\right\rvert \lesssim \eta_n \norm{f}_{C^2_b}~\rightarrow 0,   
    \end{align*}
    \begin{align*}
       &\left\lvert\int_0^t\int_{\T^3\times \R^3} \operatorname{div}_b\left(\mathcal{L}^n(b)\nabla_b f(x,b)\right)\tilde{\mu}_n(s,x,db) dx ds-\int_0^t\int_{\T^3\times \R^3} \operatorname{div}_b\left(\mathcal{L}(b)\nabla_b f(x,b)\right)\overline{\mu}(s,x,db) dx ds\right\rvert\notag\\ & \leq  \left\lvert\int_0^t\int_{\T^3\times \R^3} \operatorname{div}_b\left(\left(\mathcal{L}^n(b)-\mathcal{L}(b)\right)\nabla_b f(x,b)\right)\tilde{\mu}_n(s,x,db) dx ds\right\rvert\notag\\ & +\left\lvert\int_0^t\int_{\T^3\times \R^3} \operatorname{div}_b\left(\mathcal{L}(b)\nabla_b f(x,b)\right)\left(\overline{\mu}(s,x,db)-\tilde{\mu}_n(s,x,db)\right) dx ds\right\rvert\rightarrow 0\quad \mathbb{P}-a.s.,
    \end{align*}
    where  we  used  that $\operatorname{div}_b\left(\mathcal{L}^n(b)-\mathcal{L}(b)\right)=0$ to show the convergence of the first summand.
    Concerning the stochastic integrals we argue by Burkholder-Davis-Gundy inequality obtaining
    \begin{align*}
&\expt{\lvert\sumkj\int_0^t\int_{\T^3\times \R^3} \skj(x)\cdot(\nabla_x f)(x,b)\tilde{\mu}_n(s,x,db)dxdW^{k,j}_s\rvert^2}\\  
& =\sumkj \expt{\int_0^t\lvert\int_{\T^3\times \R^3} \skj(x)\cdot(\nabla_x f)(x,b)\tilde{\mu}_n(s,x,db)dx \rvert^2 ds}\\ & \lesssim \sum_{\substack{k\in \Z^3_0\\ n\leq \lvert k\rvert \leq 2n} } \frac{1}{\lvert k\rvert^5}\norm{\nabla_x f}_{C_b}^2\rightarrow 0.
    \end{align*}
    For the second one we need to argue in a more precise way, using the fact that $\{a_{k,j}\otimes a_{k,l} e^{ik\cdot x}\}_{\substack{k\in \Z^3_0\\ j,l\in\{1,2,3\}}}$ is an orthonormal basis of $L^2(\T^3;\R^3\otimes \R^3)$. Let us define for each $f$ and $n$ the process
    \begin{align*}
        F^n_f(t,x)=\int_{\R^3}b\otimes \nabla_b f(x,b) \tilde{\mu}_n(t,x,db).
    \end{align*}
    Due to Jensen's inequality and \autoref{well_posedness_estimates_MF} \begin{align*}
     \expt{\sup_{t\in [0,T]}\norm{F^n_f(t)}_{L^2(\T^3)}^2}&\leq \expt{\sup_{t\in [0,T]}\int_{\T^3}\int_{\R^3}\lvert b\otimes \nabla_b f(x,b) \rvert ^2\tilde{\mu}_n(t,x,db) dx   }\\ & \leq  \norm{f}_{C^1_b}^2\expt{\sup_{t\in [0,T]}\int_{\T^3}\int_{\R^3}\lvert b\rvert ^2\tilde{\mu}_n(t,x,db) dx   }\\ & \leq \norm{f}_{C^1_b}^2 C_0(\overline{B}_0,T).
    \end{align*}
    Therefore
\begin{align*}
&\expt{\lvert\sumkj\int_0^t\int_{\T^3\times \R^3} b\cdot\nabla\skj(x)\nabla_b f(x,b)\Tilde{\mu}_n(s,x,db)dxdW^{k,j}_s \rvert^2}\\ & =\sumkj\expt{\int_0^t\lvert \int_{\T^3\times \R^3} b\cdot\nabla\skj(x)\nabla_b f(x,b)\Tilde{\mu}_n(s,x,db)dx\rvert^2 ds}\\ & =    \sumkj \lvert \tkj\rvert^2\lvert k\rvert^2\expt{\int_0^t \langle \frac{k}{\lvert k\rvert}\otimes a_{k,j} e^{ik\cdot x}, F^n_f(s,x) \rangle_{L^2(\T^3)}^2 ds}\\ & \leq \frac{1}{n^3}\expt{\int_0^T \norm{F^n_f(t)}_{L^2(\T^3)}^2 dt }\\ & \leq \frac{T}{n^3}\norm{f}_{C^1_b}^2 C_0(\overline{B}_0,T)\rightarrow 0.
\end{align*}
This completes the proof of the first step. \\
\emph{Step 2: Conclusion.}
By continuity in time of the objects and separability of $C^2_c(\T^3\times \R^3)$ we can find a zero probability set $\mathcal{N}$ such that on its complementary the PDE in weak form is satisfied for each $t\in [0,T],$  $f\in C^{2}_c(\T^3\times \R^3)$. Secondly, let us consider $f\in C^{1}_c((0,T);C^2_c(\T^3\times\R^3))$ and $0=t_0<t_1<\dots<t_N=T$ a partition of $[0,T]$ also denoted by $\pi$. Thanks to the identity 
\begin{multline*}
\int_{\T^3\times \R^3} f_{t_{i+1}}(x,b)\overline{\mu}(t_{i+1},x,db)dx-\int_{\T^3\times \R^3} f_{t_{i}}(x,b)\overline{\mu}(t_{i+1},x,db)dx\\=\int_{t_{i}}^{t_{i+1}}\int_{\T^3\times \R^3} \partial_s f_{s}(x,b)\overline{\mu}(t_{i+1},x,db)dxds
\end{multline*}
and relation \eqref{limit_pde_weak_indep_time} we obtain
\begin{align*}
\int_{\T^3\times \R^3}f_{i+1}(x,b)\overline{\mu}(t_{i+1},x,db)dx&-\int_{\T^3\times \R^3}f_{t_i}(x,b)\overline{\mu}(t_{i},x,db)dx\\ &=
\int_{t_{i}}^{t_{i+1}}\int_{\T^3\times \R^3} \partial_s f_{s}(x,b)\overline{\mu}(t_{i+1},x,db)dxds\\ &+\int_{t_{i}}^{t_{i+1}}\int_{\T^3\times \R^3} \operatorname{div}_b\left(\mathcal{L}(b)\nabla_b f_{t_i}(x,b)\right)\overline{\mu}(s,x,db) dx ds,  
\end{align*}
which implies
\begin{align*}
\int_{0}^{T}\int_{\T^3\times \R^3} \partial_s f_{s}(x,b)\overline{\mu}(s^{+}_{\pi},x,db)dxds+\int_{0}^{T}\int_{\T^3\times \R^3} \operatorname{div}_b\left(\mathcal{L}(b)\nabla_b f_{s^{-}_{\pi}}(x,b)\right)\overline{\mu}(s,x,db) dx ds=0,  
\end{align*}
where $s^{-}_{\pi}=t_i$ and $s^{+}_{\pi}=t_{i+1}$ if $s\in [t_i,t_{i+1})$. Taking a sequence $\pi_n$ with size going to $0$ we conclude the proof thanks to dominated convergence's theorem and the regularity of $\overline{\mu}$ and $f$.
\end{proof}
Combining \autoref{compactness laws}, \autoref{prop:uniqueness_vlasov} and \autoref{prop_limit_PDE} we can easily prove \autoref{main_thm_magnetic_field}.
\begin{proof}[Proof of \autoref{main_thm_magnetic_field}]
We denote by $\L_n$ the law of $\mu_n$ on $C([0,T];\mathcal{Y})$.
By \autoref{compactness laws}, \autoref{prop:uniqueness_vlasov} and \autoref{prop_limit_PDE}, every subsequence $\L_{n(k)}$ admits a sub-subsequence which converges to the unique limit point $\delta_\mu$, where $\mu$ is the deterministic solution of \eqref{limit_PDE_vlasov} in the sense of \autoref{weak_vlasov_pde}. Then, for example by \cite[Theorem 2.6]{Bil1999}, the whole sequence $\L_{n}$ converges weakly to $\delta_\mu$, and then also in probability as the limit point is a Dirac delta (see e.g. the argument in \cite[page 27]{Bil1999}). The last claim follows introducing a family of functions $f^{M,i}\in C^2_b(\T^3\times \R^3)$ for $i\in \{1,2,3\},\ M\in \N$ such that
\begin{align*}
     \norm{\nabla^2 f^{M,i}}_{C_b(\T^3\times \R^3)}\leq 1,\quad  \lvert f^{M,i}(x,b)\rvert\leq \lvert b\rvert,\quad f^{M,i}(x,b)=
        b^i \quad &\text{if } \lvert b \rvert\leq M.
\end{align*}
Using these $f^{M,i}$ as test functions
in \eqref{limit_pde_weak_indep_time} and letting $M\rightarrow+\infty$, by dominated convergence theorem we get the claim. The proof is complete.    
\end{proof}
\appendix
\section{Global Well-Posedness of equation \eqref{ito_PDE_magnetic}}\label{appendix PDE magnetic Field}
The proof of \autoref{well_posedness_estimates_MF} is a consequence of a vanishing viscosity procedure. In order to ease the notation, let us fix $\chi=1$. For each $n\in \N$, given a family of positive numbers $\eta\rightarrow 0$, let us introduce the approximate problem
\begin{align}\label{ito_PDE_viscous_magnetic}
 \begin{cases}
     dB^{n,\eta}_t&=(\frac{2}{3}\eta_n+\eta)\Delta B^{n,\eta}_t dt+\sumkj (\skj\cdot\nabla B^{n,\eta}_t-B^{n,\eta}_t\cdot\nabla \skj)dW^{k,j}_t,\\
     \div B^{n,\eta}_t&=0,\\
     B^{n,\eta}_0&=\overline{B}_0.
 \end{cases}   
\end{align}  
Analogously to \autoref{weak_solution MF} we introduced a notion of weak solution for the viscous problem \eqref{ito_PDE_viscous_magnetic}.
\begin{definition}\label{viscous_magnetic_weak_sol}
A stochastic process $C_{\mathcal{F}}(0,T;V)\cap L_{\mathcal{F}}^2(0,T;\mathbf{H}^2)$ is a weak solution of equation \eqref{ito_PDE_viscous_magnetic} if $\mathbb{P}-$a.s.
\begin{align*}
 \langle B^{n,\eta}_t,\phi\rangle-\langle \overline{B}_0,\phi\rangle&=-\left(\frac{2}{3}\eta_n+\eta\right)\int_0^t \langle \nabla B^{n,\eta}_s,\nabla \phi\rangle ds\\ &+\sumkj \int_0^t \langle\skj \cdot\nabla B^{n,\eta}_s - B^{n,\eta}_s\cdot \nabla \skj ,\phi\rangle dW^{k,j}_s   
\end{align*}
for each $\phi \in V,\ t\in [0,T]$.        
\end{definition}
The well-posedness of the stochastic PDE above is a well-known fact, see for example \cite[Chapters 3-5]{Flandoli_Book_95}. Indeed the following holds
\begin{theorem}\label{well-posedness_viscous_pde}
For each $\overline{B}_0\in V$ there exists a unique weak solution of \eqref{ito_PDE_viscous_magnetic} in the sense of \autoref{viscous_magnetic_weak_sol}.    
\end{theorem}
\subsection{A priori estimates for the viscous equation}
In this section we provide for each $n\in \N$ some a priori estimates for $B^{n,\eta}$ uniformly in the viscosity parameter. 
\begin{proposition}\label{prop_compactess_space_viscous}
For each $\gamma\geq 4$ we have
\begin{align*}
    \expt{\sup_{t\in [0,T]}\norm{B^{n,\eta}_t}_V^\gamma}\lesssim_{n,\overline{B}_0,T,\gamma}1.
\end{align*}
\end{proposition}
\begin{proof}
We divide the proof of this proposition in many steps.\\
\emph{Step 1: It\^o formula.}
    By \autoref{viscous_magnetic_weak_sol}
    \begin{align*}
        \int_0^T \norm{\Delta B^{n,\eta}_t}^2 dt<+\infty \quad \mathbb{P}-a.s.
    \end{align*}
    and 
    \begin{align*}
        \sumkj \int_0^t \skj\cdot\nabla B^{n,\eta}_s-B^{n,\eta}_s\cdot\nabla\skj dW^{k,j}_s
    \end{align*}
    is a $V$ valued local martingale. Therefore we can apply It\^o formula in $V$, see \cite[Theorem 2.13]{linear_theory_rozovsky_lotosky}, obtaining
    \begin{align*}
        d\norm{\nabla B^{n,\eta}_t}^2+2\left(\frac{2}{3}\eta_n+\eta\right) \norm{\Delta B^{n,\eta}_t}^2dt&=dM^{n,\eta}_t+I^{n,\eta}_tdt
    \end{align*}
    where
    \begin{align*}
        dM^{n,\eta}_t&:=2\sumkj \langle \nabla (\skj\cdot\nabla B^{n,\eta}_t-B^{n,\eta}_t\cdot\nabla\skj),\nabla B^{n,\eta}_t\rangle dW^{k,j}_t,\\
        I^{n,\eta}_t&:=2\sumkj \langle \nabla (\skj\cdot\nabla B^{n,\eta}_t-B^{n,\eta}_t\cdot\nabla\skj), \nabla (\smkj\cdot\nabla B^{n,\eta}_t-B^{n,\eta}_t\cdot\nabla\smkj) \rangle.
    \end{align*}
Since $\skj$ are divergence free the second order term in $B^{n,\eta}$ cancels, i.e.
\begin{align*}
    \langle \skj\cdot \nabla^2 B^{n,\eta}_t,\nabla B^{n,\eta}_t\rangle = 0. 
\end{align*}
Therefore 
\begin{align*}
    dM^{n,\eta}_t=2\sumkj\langle \nabla \skj \cdot\nabla B^{n,\eta}_t-\nabla B^{n,\eta}_t\cdot\nabla \skj-B^{n,\eta}_t\cdot\nabla^2\skj,\nabla B^{n,\eta}_t\rangle dW^{k,j}_t.
\end{align*}
Let us write better also $I^{n,\eta}_t$
\begin{align*}
I^{n,\eta}_t&=I^{0,n,\eta}_t+I^{1,n,\eta}_t+I^{2,n,\eta}_t+I^{3,n,\eta}_t+I^{4,n,\eta}_t,
\end{align*}
where
\begin{align*}
    I^{0,n,\eta}_t&=2\sumkj\langle \nabla\skj\cdot\nabla B^{n,\eta}_t, \nabla\smkj\cdot\nabla B^{n,\eta}_t-\nabla B^{n,\eta}_t\cdot\nabla \smkj\rangle\\ 
    & -2\sumkj \langle \nabla B^{n,\eta}_t\cdot\nabla\skj,\nabla\smkj\cdot\nabla B^{n,\eta}_t-\nabla B^{n,\eta}_t\cdot\nabla \smkj\rangle
    \\ & +2\sumkj\langle B^{n,\eta}_t\cdot \nabla^2 \skj,B^{n,\eta}_t\cdot\nabla^2 \smkj\rangle,\\
    I^{1,n,\eta}_t&=-2\sumkj\left(\langle \skj\cdot\nabla^2 B^{n,\eta}_t ,B^{n,\eta}_t\cdot\nabla \smkj\rangle+\langle B^{n,\eta}_t\cdot\nabla \skj,\smkj\cdot\nabla^2 B^{n,\eta}_t, \rangle\right)\\
    I^{2,n,\eta}_t&=2\sumkj\langle \skj\cdot\nabla^2 B^{n,\eta}_t,\nabla\smkj\cdot\nabla B^{n,\eta}_t-\nabla B^{n,\eta}_t\cdot\nabla\smkj\rangle\\ & +2\sumkj\langle \nabla\skj\cdot\nabla B^{n,\eta}_t-\nabla B^{n,\eta}_t\cdot\nabla\skj,\smkj\cdot\nabla^2 B^{n,\eta}_t\rangle,\\
    I^{3,n,\eta}_t&=-2\sumkj \langle \nabla\skj\cdot\nabla B^{n,\eta}_t-\nabla B^{n,\eta}_t\cdot\nabla\skj, B^{n,\eta}_t\cdot\nabla^2\smkj\rangle\\ & -2\sumkj \langle \nabla\smkj\cdot\nabla B^{n,\eta}_t-\nabla B^{n,\eta}_t\cdot\nabla\smkj, B^{n,\eta}_t\cdot\nabla^2\skj\rangle
    \\
    I^{4,n,\eta}_t&=2\sumkj \langle \skj\cdot\nabla^2 B^{n,\eta}_t, \smkj\cdot\nabla^2 B^{n,\eta}_t \rangle.
\end{align*}
Now we show that $I^{1,n,\eta}_t=I^{2,n,\eta}_t=I^{3,n,\eta}_t=0,\ I^{4,n,\eta}_t=\frac{4}{3}\norm{\Delta B^{n,\eta}_t}^2$.
Expanding the definition of each summand in $I^{1,n,\eta}_t$ we have
\begin{align*}
\langle \skj\cdot\nabla^2 B^{n,\eta}_t ,B^{n,\eta}_t\cdot\nabla \smkj\rangle&+\langle B^{n,\eta}_t\cdot\nabla \skj,\smkj\cdot\nabla^2 B^{n,\eta}_t, \rangle\\ &=-i\lvert \tkj\rvert^2\sum_{\alpha,\beta,\gamma,\theta=1}^3\int_{\T^3} (a_{k,j})_{\alpha}\partial_{\alpha,\beta}(B^{n,\eta}_t)_{\gamma}(B^{n,\eta}_t)_{\theta}(k)_{\theta}(a_{k,j})_{\gamma} dx\\ & +i\lvert \tkj\rvert^2\sum_{\alpha,\beta,\gamma,\theta=1}^3\int_{\T^3} (a_{k,j})_{\alpha}\partial_{\alpha,\beta}(B^{n,\eta}_t)_{\gamma}(B^{n,\eta}_t)_{\theta}(k)_{\theta}(a_{k,j})_{\gamma} dx=0.    
\end{align*}
The analysis of $I^{2,n,\eta}_t,\ I^{3,n,\eta}_t$ is analogous and we omit the detailed computations. Concerning $I^{4,n,\eta}_t$, we have, recalling that $\sumkj \lvert \tkj\rvert^2 (a_{k,j})_{\alpha}(a_{k,j})_{\beta}=\frac{2}{3}\eta_n\delta_{\alpha,\beta}$,
\begin{align*}
I^{4,n,\eta}_t&=\frac{4}{3}\eta_n \sum_{\alpha,\beta,\gamma=1}^3 \langle \partial_{\alpha,\beta} (B^{n,\eta}_t)_{\gamma},\partial_{\alpha,\beta} (B^{n,\eta}_t)_{\gamma}\rangle\\ & =\frac{4}{3}\eta_n \norm{\Delta B^{n,\eta}_t}^2.    
\end{align*}
In conclusion, so far we showed that
\begin{align}\label{ito_viscous}
d\norm{\nabla B^{n,\eta}_t}^2+2\eta \norm{\Delta B^{n,\eta}_t}^2dt&=dM^{n,\eta}_t+I^{0,n,\eta}_tdt.    
\end{align}
Therefore, applying finite dimensional It\^o formula to relation \eqref{ito_viscous} we obtain for $\gamma\geq 4$
\begin{align}\label{ito_further_powers_viscous}
    d&\norm{\nabla B^{n,\eta}_t}^\gamma+\eta\gamma\norm{\nabla B^{n,\eta}_t}^{\gamma-2}\norm{\Delta B^{n,\eta}_t}^2dt =\frac{\gamma}{2}\norm{\nabla B^{n,\eta}_t}^{\gamma-2} I^{0,n,\eta}_t dt+\frac{\gamma}{2}\norm{\nabla B^{n,\eta}_t}^{\gamma-2}dM^{n,\eta}_t\notag\\ &+\gamma(\gamma-2)\norm{\nabla B^{n,\eta}_t}^{\gamma-4}\sumkj\left\lvert\langle \nabla \skj \cdot\nabla B^{n,\eta}_t-\nabla B^{n,\eta}_t\cdot\nabla \skj-B^{n,\eta}_t\cdot\nabla^2\skj,\nabla B^{n,\eta}_t\rangle\right\rvert^2 dt.
\end{align}
\emph{Step 2: Proof of the estimates.}
The reason why expressions \eqref{ito_viscous} and \eqref{ito_further_powers_viscous} allow to obtain estimates independent in the viscosity is due to the fact do not appear second order derivatives in $B^{n,\eta}_t$ in each term in the right hand side. For each $M\geq 1,\ n,\ \eta$ let $\tau_M=\inf\{t\in [0,T]:\ \norm{\nabla B^{n,\eta}_t}^2\geq M\}\wedge T$.
Integrating \eqref{ito_viscous} between $0$ and $r\leq \tau^M$ and taking the expected value of the supremum in time for $r\leq t\wedge \tau_M$ we obtain, by Burkholder-Davis-Gundy inequality and H\"older inequality, since only finite $\skj$ are different from $0$
\begin{align*}
    \expt{\sup_{r\leq t}\norm{\nabla B^{n,\eta}_{r \wedge \tau_M}}^2}&=\expt{\sup_{r\leq t\wedge \tau_M}\norm{\nabla B^{n,\eta}_r}^2}\\ & \leq \norm{\nabla \overline{B}_0}^2+\expt{\int_0^{t\wedge \tau_M}\lvert I^{0,n,\eta}_s\rvert ds}+\expt{\left(\int_0^{t\wedge \tau_M}d[M^{n,\eta},M^{n,\eta}]_s ds\right)^{1/2}}\\ &\leq \norm{\nabla \overline{B}_0}^2+C_{n}\expt{\int_0^{t\wedge\tau_M}\norm{\nabla B^{n,\eta}_s}^2 ds}\\ &\quad\quad\quad\qquad+C_{n}\expt{\left(\int_0^{t\wedge\tau_M}\norm{\nabla B^{n,\eta}_s}^4 ds\right)^{1/2}}\\ & \leq \norm{\nabla \overline{B}_0}^2+C_{n}\expt{\int_0^{t\wedge\tau_M}\norm{\nabla B^{n,\eta}_s}^2 ds}\\ &\quad\quad\quad\qquad+C_{n}\expt{\sup_{r\leq t\wedge\tau_M}\norm{\nabla B^{n,\eta}_r}\left(\int_0^{t\wedge\tau_M}\norm{\nabla B^{n,\eta}_s}^2 ds\right)^{1/2}}\\ & \leq \norm{\nabla \overline{B}_0}^2+\frac{1}{2}\expt{\sup_{r\leq t\wedge\tau_M}\norm{\nabla B^{n,\eta}_r}^2}+C_{n}\expt{\int_0^{t\wedge\tau_M}\norm{\nabla B^{n,\eta}_s}^2 ds},
\end{align*}
where the last step follows by Young's inequality. In conclusion, by simple manipulations, we proved that
\begin{align*}
    \expt{\sup_{r\leq t}\norm{\nabla B^{n,\eta}_{r \wedge \tau_M}}^2}\leq 2\norm{\nabla \overline{B}_0}^2+C_n\expt{\int_0^t \operatorname{sup}_{r\leq s} \norm{\nabla B^{n,\eta}_{r \wedge \tau_M}}^2 ds}.
\end{align*}
By Gr\"onwall's lemma, the latter implies
\begin{align*}
    \expt{\sup_{r\leq T\wedge \tau_M}\norm{\nabla B^{n,\eta}_r}^2}=\expt{\sup_{r\leq T}\norm{\nabla B^{n,\eta}_{r \wedge \tau_M}}^2}\lesssim_{n,\overline{B}_0,T}1.
\end{align*}
Since $\tau_M\rightarrow T$ as $M\rightarrow +\infty$ due to the fact that $B^{n,\eta}\in C([0,T];V)$, letting $M\rightarrow +\infty$ the latter implies the claim for $\gamma=2$ by monotone convergence. The proof of the claim for $\gamma\geq 4$ is analogous, only starting by \eqref{ito_further_powers_viscous} in place of \eqref{ito_viscous} and we omit the easy details.
\end{proof}
\begin{proposition}\label{prop_compactness_time_viscius}
For each $\beta\geq 4,s_1\in (0,\frac{1}{2}),r_1>2$ such that $r_1s_1>1$ we have
\begin{align*}
    \expt{\int_0^T\int_0^T \frac{\norm{B^{n,\eta}_t-B^{n,\eta}_s}_{\mathbf{H}^{-\beta}}^{r_1}}{\lvert t-s\rvert^{1+r_1s_1}}dt ds}\lesssim_{n,\overline{B}_0,\beta,r_1,s_1,T} 1.
\end{align*}
\end{proposition}
\begin{proof}
We start estimating $ \expt{\lvert\langle B^{n,\eta}_t-B^{n,\eta}_s,a_{h,l}e^{ih\cdot x}\rangle\rvert^{r_1}}$ for each $h\in \Z^3_0,\ l\in \{1,2\}$.
Thanks to the weak formulation satisfied by $B^{n,\eta}_t$ we have
\begin{align*}
 \langle B^{n,\eta}_t-B^{n,\eta}_s,a_{h,l}e^{ih\cdot x}\rangle&=\left(\frac{2}{3}\eta_n+\eta\right)\int_s^t \langle \nabla B^{n,\eta}_r,h\otimes a_{h,l}e^{ih\cdot x}\rangle dr\\ &+\sumkj \int_s^t \langle\skj \cdot\nabla B^{n,\eta}_r - B^{n,\eta}_r\cdot \nabla \skj ,a_{h,l}e^{ih\cdot x}\rangle dW^{k,j}_r.   
\end{align*}    
Taking the modulus of the expression above and raising to the power $r_1$ we can estimate the quantity required by Burkholder-Davis-Gundy inequality. Indeed, thanks to H\"older inequality and \autoref{prop_compactess_space_viscous}
\begin{align}\label{step_1_compt_time_viscous}
\expt{\lvert\langle B^{n,\eta}_t-B^{n,\eta}_s,a_{h,l}e^{ih\cdot x}\rangle\rvert^{r_1}}&\lesssim_{r_1,n}\lvert h\rvert^r\expt{\left(\int_s^t\norm{\nabla B^{n,\eta}_r}dr\right)^{r_1}}+\expt{\left(\int_s^t \norm{\nabla B^{n,\eta}_r}^{2}dr\right)^{r_1/2}}\notag\\ & \leq \lvert h\rvert^{r_1}\lvert t-s\rvert^{r_1}\expt{\sup_{r\in [0,T]}\norm{\nabla B^{n,\eps}_r}^{r_1}}+\lvert t-s\rvert^{\frac{r_1}{2}}\expt{\sup_{r\in [0,T]}\norm{\nabla B^{n,\eps}_r}^{r_1}} \notag \\ & \lesssim_{n,\overline{B}_0,T,r_1} \lvert h\rvert^{r_1}\lvert t-s\rvert^{\frac{r_1}{2}}. 
\end{align}
Secondly, by definition of $\mathbf{H}^{-\beta}$ norm and H\"older's inequality for some $\eps>0$ small enough
\begin{align}\label{step_2_compt_time_viscous}
\expt{\norm{B^{n,\eta}_t-B^{n,\eta}_s}^{r_1}_{\mathbf{H}^{-\beta}}}&=\expt{\left(\sum_{\substack{h\in \Z^3_0\\ l\in \{1,2\}}}\frac{\langle B^{n,\eta}_t-B^{n,\eta}_s,a_{h,l}e^{ih\cdot x} \rangle^2}{\lvert h\rvert^{2\beta}} \right)^{\frac{r_1}{2}}} \notag  \\ & =\expt{\left(\sum_{\substack{h\in \Z^3_0\\ l\in \{1,2\}}}\frac{\langle B^{n,\eta}_t-B^{n,\eta}_s,a_{h,l}e^{ih\cdot x} \rangle^2}{\lvert h\rvert^{\frac{3(1+\eps)(r_1-2)}{r_1}}\lvert h\rvert^{\frac{2\beta r_1-3(1+\eps)(r_1-2)}{r_1}}} \right)^{\frac{r_1}{2}}}\notag\\ & \leq \left(\sum_{\substack{h\in \Z^3_0\\ l\in \{1,2\}}}\frac{1}{\lvert h\rvert^{3(1+\eps)}}\right)^{\frac{r_1-2}{2}}   \expt{\sum_{\substack{h\in \Z^3_0\\ l\in \{1,2\}}}\frac{\langle B^{n,\eta}_t-B^{n,\eta}_s,a_{h,l}e^{ih\cdot x} \rangle^{r_1}}{\lvert h\rvert^{\beta r_1-\frac{3(1+\eps)(r_1-2)}{2}}} }\notag\\ & \lesssim_{r_1}\expt{\sum_{\substack{h\in \Z^3_0\\ l\in \{1,2\}}}\frac{\langle B^{n,\eta}_t-B^{n,\eta}_s,a_{h,l}e^{ih\cdot x} \rangle^{r_1}}{\lvert h\rvert^{\beta r_1 -\frac{3(1+\eps)(r_1-2)}{2}}} }.
\end{align}
Combining \eqref{step_1_compt_time_viscous} and \eqref{step_2_compt_time_viscous} we get immediately the claim of the proposition. Indeed
\begin{multline*}
    \expt{\int_0^T\int_0^T \frac{\norm{B^{n,\eta}_t-B^{n,\eta}_s}_{\mathbf{H}^{-\beta}}^{r_1}}{\lvert t-s\rvert^{1+r_1s_1}}dt ds}\\ \lesssim_{n,\overline{B}_0,T,r_1}\sum_{\substack{h\in \Z^3_0\\ l\in \{1,2\}}}\frac{1}{\lvert h\rvert^{(\beta-1)r_1-\frac{3(1+\eps)(r_1-2)}{2}}}\int_0^T\int_0^T\frac{1}{\lvert t-s\rvert^{1+r_1(s_1-\frac{1}{2})}}dtds.
\end{multline*}
The latter being finite as soon as $s_1\in (0,\frac{1}{2})$ and
\begin{align*}
 (\beta-1)r_1-\frac{3(1+\eps)(r_1-2)}{2}>3.   
\end{align*}
The relation above is always satisfied if $\beta\geq 4$ for each choice of $r_1>2 $ and $\eps\in (0,1)$.
\end{proof}
\subsection{Tightness and limit object}
Let us denote by \begin{align*}
    \mathbb{B}_R=\{x\in V: \norm{x}_V\leq R\}
\end{align*}
and $q$ the metric on $\mathbb{B}_R$ compatible with the weak topology. Let us introduce the space
\begin{align*}
    C([0,T];\mathbb{B}_R)=\{f\in C([0,T];V_w):\ \norm{f(t)}_V\leq R\quad \forall t\in [0,T]\}.
\end{align*}
The space $C([0,T];\mathbb{B}_R)$ is metrizable and complete if endowed with the metric
\begin{align*}
d(u,v)=\sup_{t\in [0,T]}q(u(t),v(t))    
\end{align*}
see \cite{Brez_book_f}, 
\cite{brzezniak2013existence}, \cite{brzezniak2021stochastic}.
Secondly we denote by \begin{align*}
    Z_T=C([0,T];H)\cap C([0,T];V_w).
\end{align*}
endowed with the supremum of the corresponding topologies, i.e. the coarsest topology on $Z_T$ that is finer either than the one in $C([0,T];H)$ and the one in $C([0,T];V_w).$ 
We are interested to introduce compact sets in $Z_T$. For this reason we recall the following results, see \cite[Lemma 2.1]{brzezniak2014existence} and \cite[Corollary 5]{simon1986compact}.
\begin{theorem}\label{lemma_abstract_convergence_1}
 Let $u_k\in C([0,T];V_w)$ such that\begin{align*}
     \sup_k \sup_{t\in [0,T]}\norm{ u_k(t)}_V\leq R,\quad u_k\rightarrow u\quad\text{in } C([0,T];H).
 \end{align*}
 Then $u_k,u \in C([0,T];\mathbb{B}_r),\ u_k\rightarrow u $ in $C([0,T];\mathbb{B}_r)$.
\end{theorem}
\begin{theorem}\label{lemma_compactness_1}
    Let $p,r \in [1,+\infty]$ and $s\in \mathbb R$; assume that $s>0$ if $r\geq p$ or $s>1/r-1/p$ if $r\leq p$.
    Let $X,\ B,\ Y$ be separable Banach spaces such that \begin{align*}
    X\stackrel{c}{\hookrightarrow}B\hookrightarrow Y;
\end{align*}
    and $F$ a bounded subset in $L^p(0,T;X)\cap W^{s,r}(0,T;Y)$. Then $F$ is relatively compact in $L^p(0,T;B)$ (in $C([0,T];B)$ if $p=+\infty$).
\end{theorem}
Combining \autoref{lemma_abstract_convergence_1} and \autoref{lemma_compactness_1} we obtain the following.
\begin{lemma}\label{lemma_compactness}
Let $s\in (0,\frac{1}{2}),\ r>2$ such that $sr>1$ and $F$ a bounded subset in $C(0,T;V)\cap W^{s,r}(0,T;\mathbf{H}^{-4})$. Then $F$ is relatively compact in $Z_T$.     
\end{lemma}
\begin{proof}
Since $F$ is bounded in $C([0,T];V),$ there exists $R<+\infty$ such that
\begin{align*}
    \sup_{f\in F}\norm{f}_{C([0,T];V)}\leq R.
\end{align*}
Therefore we can consider on $Z_T$ the metric of $C([0,T];\mathbb{B}_R)\cap C([0,T];H)$. In particular compactness is equivalent to sequentially compactness. Let $(f_k)\in F$. Thanks to \autoref{lemma_compactness_1} with $p=+\infty,\ X=V,\ B=H,\ Y=\mathbf{H}^{-4}$ there exists a subsequence $k_j$ and $f\in C([0,T];H)$ such that \begin{align*}
    f_{k_j}\rightarrow f\quad \text{in }C([0,T];H).
\end{align*}
Then \autoref{lemma_abstract_convergence_1} implies that also 
\begin{align*}
    f_{k_j}\rightarrow f\quad \text{in }C([0,T];\mathbb{B}_R).
\end{align*}
This completes the proof.
\end{proof}
Having  \autoref{prop_compactess_space_viscous}, \autoref{prop_compactness_time_viscius}, \autoref{lemma_compactness} in mind we get immediately the following:
\begin{corollary}\label{compactness_viscous}
For each $n\in \mathbb{N}$ the family of laws $\{\L(B^{n,\eta})\}_{\eta > 0}$ is tight in $Z_T$.    
\end{corollary}
By standard arguments, see for example \cite[Chapter 2]{flandoli2023stochastic} or \cite[Section 5.3]{brzezniak2021stochastic}, by Jakubowski version of Skorokhod's representation's theorem, \cite[Theorem 2]{yakubovskii1997almost}, for each $n\in \N$ we can find, up to passing to subsequences, an auxiliary probability space, that for simplicity we continue to call $(\Omega,\mathcal{F},\mathbb{P})$, and processes $(\Tilde{B}^{n,\eta_h},W^{h}:=\{W^{k,j,h}\}_{k\in \mathbb{Z}^3_0,\ j\in \{1,2\}}),\  (B^n,W:=\{W^{k,j}\}_{k\in \mathbb{Z}^3_0,\ j\in\{1,2\}}),  $ such that \begin{align*}
    &\Tilde{B}^{n,\eta_h}\rightarrow B^n\quad \text{in } C([0,T];V_w)\cap C([0,T];H) \quad \mathbb{P}-a.s.\\
   & W^h\rightarrow W \quad \text{in } C([0,T];\mathbb{C}^{\mathbb{Z}^3_0\times \mathbb{Z}^3_0})\quad \mathbb{P}-a.s.
\end{align*}
Of course the convergence above between $W^n$ and $W$ can be seen as the uniform convergence of cylindrical Wiener processes \begin{align*}
    W^h=\sumkj a_{k,j}e^{ik\cdot x} W^{k,j,h},\ W=\sumkj a_{k,j}e^{ik\cdot x} W^{k,j}
\end{align*} on a suitable Hilbert space $U_0$. Before going on, in order to identify $B^n$ as a weak solution of equation \eqref{ito_PDE_magnetic} we need further integrability properties of $B^n$.
\begin{proposition}\label{further_estimates_limit_solution}
For each $\gamma\geq 4$
\begin{align*}
    \expt{\sup_{t\in [0,T]}\norm{B^{n}_t}_V^\gamma}\lesssim_{n,\overline{B}_0,T,\gamma}1.
\end{align*}
\end{proposition}
\begin{proof}
     Since $\tilde{B}^{n,\eta_h}$ has the same law of $B^{n,\eta_h}$, by \autoref{prop_compactess_space_viscous} for each $\gamma\geq 4$
\begin{align*}
    \sup_{h\in \N}\expt{\sup_{t\in [0,T]}\norm{\nabla \tilde{B}^{n,\eta_h}}^\gamma}\lesssim_{n,\overline{B}_0,T,\gamma}1.
\end{align*}
In particular up to passing to a further subsequence $B^{n,\eta,h_l}$ converges weakly star in $L^{\gamma}(\Omega;L^{\infty}(0,T;V))$ to some $\overline{B}^n$. The latter implies also that the weak star convergence holds in $L^{\gamma}(\Omega;L^{\infty}(0,T;H))$. Since $B^{n,\eta,h_l}\rightarrow B^n\ \mathbb{P}-a.s.$ in $C([0,T];H)\hookrightarrow L^{\infty}(0,T;H)$ we have that $B^n=\overline{B}^n$ and in particular the claim.
\end{proof}
Now we are in the position to prove \autoref{well_posedness_estimates_MF}.
\begin{proof}[Proof of \autoref{well_posedness_estimates_MF}]
We divide the proof in some steps.\\
\emph{Step 1: A priori estimates.}
We start showing that each solution of \eqref{ito_PDE_magnetic} in the sense of  \autoref{weak_solution MF} satisfies \eqref{ito2magnetic}, \eqref{itogammamagnetic}, \eqref{apriori_estimateB_2}, \eqref{apriori_estimate_gamma}.
 By \autoref{weak_solution MF}
    \begin{align*}
        \int_0^T \norm{\Delta B^{n}_t}_{\mathbf{H}^{-1}}^2 dt<+\infty \quad \mathbb{P}-a.s.
    \end{align*}
    and 
    \begin{align*}
        \sumkj \int_0^t \skj\cdot\nabla B^{n}_s-B^{n}_s\cdot\nabla\skj dW^{k,j}_s
    \end{align*}
    is a $H$ valued local martingale. Therefore we can apply It\^o formula in $H$, see \cite[Theorem 2.13]{linear_theory_rozovsky_lotosky}, obtaining \eqref{ito2magnetic}.
Therefore, applying finite dimensional It\^o formula to relation \eqref{ito2magnetic} with $f(x)=\lvert x\rvert^{\gamma/2}$ for $\gamma\geq 4$  we obtain \eqref{itogammamagnetic}. For each $M\geq 1,\ n\in\mathbb{N}$ let $\tau_M=\inf\{t\in [0,T]:\ \norm{\nabla B^{n}_t}^2\geq M\}\wedge T$.
Integrating \eqref{ito2magnetic} between $0$ and $r\leq \tau^M$ and taking the expected value of the supremum in time for $r\leq t\wedge \tau_M$ we obtain, by Burkholder-Davis-Gundy inequality and H\"older inequality
\begin{align*}
    \expt{\sup_{r\leq t}\norm{ B^{n}_{r \wedge \tau_M}}^2}&=\expt{\sup_{r\leq t\wedge \tau_M}\norm{ B^{n}_r}^2}\\ & \leq \norm{ \overline{B}_0}^2+\sumkj \lvert k\tkj\rvert^2\expt{\int_0^{t\wedge \tau_M}\lVert B^n_s\rVert^2 ds}\\ &+\expt{\left(\int_0^{t\wedge \tau_M}\sumkj \lvert k\tkj\rvert^2  \lVert B^n_s\rVert^4 ds\right)^{1/2}}\\ &\leq \norm{ \overline{B}_0}^2+C\expt{\int_0^{t\wedge\tau_M}\norm{ B^{n}_s}^2 ds}+C\expt{\left(\int_0^{t\wedge\tau_M}\norm{B^{n}_s}^4 ds\right)^{1/2}}\\ & \leq \norm{\overline{B}_0}^2+C\expt{\int_0^{t\wedge\tau_M}\norm{ B^{n}_s}^2 ds}\\ &\quad\quad\quad\quad\qquad+C\expt{\sup_{r\leq t\wedge\tau_M}\norm{ B^{n}_r}\left(\int_0^{t\wedge\tau_M}\norm{ B^{n}_s}^2 ds\right)^{1/2}}\\ & \leq \norm{ \overline{B}_0}^2+\frac{1}{2}\expt{\sup_{r\leq t\wedge\tau_M}\norm{B^{n}_r}^2}+C\expt{\int_0^{t\wedge\tau_M}\norm{ B^{n}_s}^2 ds},
\end{align*}
where in the second step we exploited the fact that
\begin{align*}
    \sumkj \lvert k\tkj\rvert^2=\sum_{\substack{k\in \Z^3_0\\
    n\leq \lvert k\rvert\leq 2n\\
    j\in \{1,2\}}}\frac{1}{\lvert k\rvert^3}\lesssim 1
\end{align*} and in the last one we applied Young's inequality. In conclusion, by simple manipulations, we proved that
\begin{align}\label{gronwall_preineq_uniqueness}
    \expt{\sup_{r\leq t}\norm{ B^{n}_{r \wedge \tau_M}}^2}\leq 2\expt{\norm{\overline{B}_0}^2}+C\expt{\int_0^t \operatorname{sup}_{r\leq s} \norm{ B^{n}_{r \wedge \tau_M}}^2 ds}.
\end{align}
By Gr\"onwall's lemma, the latter implies
\begin{align*}
    \expt{\sup_{r\leq T\wedge \tau_M}\norm{ B^{n}_r}^2}=\expt{\sup_{r\leq T}\norm{ B^{n}_{r \wedge \tau_M}}^2}\lesssim_{\overline{B}_0,T}1.
\end{align*}
Since $\tau_M\rightarrow T$ as $M\rightarrow +\infty$ due to the fact that $B^{n}\in C([0,T];H)$, letting $M\rightarrow +\infty$ the latter implies equation \eqref{apriori_estimateB_2} by monotone convergence. The proof of equation \eqref{apriori_estimate_gamma} is analogous, only starting by \eqref{itogammamagnetic} in place of \eqref{ito2magnetic} and we omit the easy details.\\
\emph{Step 2: Uniqueness.} By linearity of the equation it is enough to prove uniqueness in case of $\overline{B}_0=0$. In particular, due \eqref{gronwall_preineq_uniqueness}, our solution starting from $\overline{B}_0=0$ satisfies
\begin{align*}
    \expt{\sup_{r\leq t}\norm{ B^{n}_{r \wedge \tau_M}}^2}\leq C\expt{\int_0^t \operatorname{sup}_{r\leq s} \norm{ B^{n}_{r \wedge \tau_M}}^2 ds}.
\end{align*}
By Gr\"onwall's lemma, the latter implies
\begin{align*}
    \expt{\sup_{r\leq T\wedge \tau_M}\norm{ B^{n}_r}^2}=\expt{\sup_{r\leq T}\norm{ B^{n}_{r \wedge \tau_M}}^2}=0.
\end{align*}
Since $\tau_M\rightarrow T$ as $M\rightarrow +\infty$ due to the fact that $B^{n}\in C([0,T];H)$, letting $M\rightarrow +\infty$ the $B^n\equiv 0$ and the uniqueness follows.\\
\emph{Step 3: Limit equation.} 
Let $\phi \in \mathbf{H}^2$ we know that
\begin{align*}
    \langle \tilde{B}^{n,\eta_h}_t,\phi\rangle-\langle \overline{B}_0,\phi\rangle&=\left(\frac{2}{3}\eta_n+\eta_{h}\right)\int_0^t \langle \tilde{B}^{n,\eta_h}_s,\Delta\phi\rangle ds\\ & +\sumkj \int_0^t\langle \skj\times B^{n,\eta_h}_s,\nabla\times \phi\rangle dW^{k,j,h}_s\quad \mathbb{P}-a.s.
\end{align*}
Therefore we will show, up to passing to a further subsequence, $\mathbb{P}$-a.s. convergence of all the terms appearing above, uniformly in time.
Thanks to the fact that $\tilde{B}^{n,\eta_h}\rightarrow B^n\ \mathbb{P}-a.s.$ in $C([0,T];H)$ we get easily
\begin{align*}
    &\operatorname{sup}_{t\in [0,T]}\lvert\langle \tilde{B}^{n,\eta_h}_t-B^n_t,\phi\rangle\rvert\rightarrow 0\quad \mathbb{P}-a.s.,\\
    &\sup_{t\in [0,T]}\left\lvert\left(\frac{2}{3}\eta_n+\eta_{h}\right)\int_0^t \langle \tilde{B}^{n,\eta_h}_s,\Delta\phi\rangle ds- \frac{2}{3}\eta_n \int_0^t \langle \tilde{B}^{n}_s,\Delta\phi\rangle ds\right\rvert\rightarrow 0\quad \mathbb{P}-a.s.
\end{align*}
We are left to show the convergence of the stochastic integrals. Since 
\begin{align*}
   \sumkj \int_0^T\langle \skj\times \left(B^{n,\eta_h}_s-B^n_s\right),\nabla\times \phi\rangle^2 ds& \leq \sum_{\substack{k\in \Z^3_0\\
   n\leq \lvert k\rvert \leq 2n\\ j\in\{1,2\}}}\frac{1}{\lvert k\rvert^5}\norm{\phi}_V^2 T\norm{B^{n,\eta_h}-B^n}_{C([0,T];H)}^2\\ & \rightarrow 0, 
\end{align*}
we can apply \cite[Lemma 4.3]{MarioMarco} obtaining that 
\begin{align*}
    \sup_{t\in [0,T]}\left\lvert  \sumkj \int_0^t\langle \skj\times B^{n,\eta_h}_s,\nabla\times \phi\rangle dW^{k,j,h}_s - \int_0^t\langle \skj\times B^{n}_s,\nabla\times \phi\rangle dW^{k,j}_s\right\rvert\rightarrow 0 \quad \text{in probability},
\end{align*}
therefore $\mathbb{P}-$a.s. up to passing to a further sub-subsequence.
In conclusion, so far we showed that for each $\phi\in \mathbf{H}^2$
\begin{align}\label{final_expression_magn_pde}
 \langle B^n_t,\phi\rangle-\langle \overline{B}_0,\phi\rangle&=\frac{2}{3}\eta_n\int_0^t \langle  B^n_s,\Delta \phi\rangle ds+\sumkj \int_0^t \langle \skj\times B^n_s ,\nabla\times\phi\rangle dW^{k,j}_s   \quad \mathbb{P}-a.s.\quad \forall t\in [0,T].
\end{align}
By standard density argument, since $B^n\in C([0,T];V_w)$, we can find a zero measure set $\mathcal{N}\subseteq\Omega$ such that on its complementary relation \eqref{final_expression_magn_pde} holds for each $\phi\in V$. We omit the easy details.\\
\emph{Step 4: Conclusion.} Every subsequence $\L(B^{n,\eta_h})$ admits a sub-subsequence which converges to the unique limit point $\L(B^n)$, where $B^n$ is the unique weak solution of \eqref{ito_PDE_magnetic} in the sense of \autoref{weak_solution MF}. Then, the Gyongy–Krylov Convergence Criterion, see \cite{gyongy1996existence}, applies and the whole sequence $B^{n,\eta_h}$ converges in probability in $C([0,T];H)\cap L^2(0,T;V)$ to $B^n$ in the original probability space. The proof is complete.
\end{proof}
\section{Some Remarks on the Occupation Measure}\label{appendix:rmk_physics}
Let $f:\mathbb{T}^{3}\rightarrow\mathbb{R}^{3}$ be a measurable function.
Given $x_{0}\in\mathbb{T}^{3}$ and $\epsilon>0$, call $\nu_{x_{0},\epsilon}$
the probability measure on Borel sets of $\mathbb{R}^{3}$ defined as%
\[
\int_{\mathbb{R}^{k}}\varphi\left(  y\right)  \nu_{x_{0},\epsilon}\left(
dy\right)  =\frac{1}{\left\vert B\left(  x_{0},\epsilon\right)  \right\vert
}\int_{B\left(  x_{0},\epsilon\right)  }\varphi\left(  f\left(  x\right)
\right)  dx
\]
for every $\varphi\in C_{b}\left(  \mathbb{R}^{3}\right)  $, where $B\left(
x_{0},\epsilon\right)  $ is the ball in $\mathbb{T}^{3}$ of center $x_{0}$ and
radius $\epsilon$ and $\left\vert B\left(  x_{0},\epsilon\right)  \right\vert
$ is its Lebesgue measure. In other words, $\nu_{x_{0},\epsilon}$ is the push
forward of the normalized Lebesgue measure $\frac{1}{\left\vert B\left(
x_{0},\epsilon\right)  \right\vert }\mathcal{L}_{3}|_{B\left(  x_{0}%
,\epsilon\right)  }$ on $B\left(  x_{0},\epsilon\right)  $ under the map $f$;
or the image law of the random variable $f$ considered on the probability
space $\left(  B\left(  x_{0},\epsilon\right)  ,\mathcal{B}\left(  B\left(
x_{0},\epsilon\right)  \right)  ,\frac{1}{\left\vert B\left(  x_{0}%
,\epsilon\right)  \right\vert }\mathcal{L}_3|_{B\left(  x_{0},\epsilon
\right)  }\right)  $. \textit{We call }$\nu_{x_{0},\epsilon}$\textit{ the
local occupation measure around }$x_{0}$\textit{ of size }$\epsilon$\textit{.
}

Clearly, if $f$ is continuous, the weak limit as $\epsilon\rightarrow0$ of
$\nu_{x_{0},\epsilon}$ is $\delta_{f\left(  x_{0}\right)  }$. When $f$ is
rapidly oscillating, however, $\nu_{x_{0},\epsilon}$ \textit{captures the
local oscillations in a statistical sense}, for finite $\epsilon$. Here
"statistical" is understood in a spatial sense (opposite to the more common
temporal or probabilistic sense):\ where (on $\mathbb{R}^{3}$) $\nu
_{x_{0},\epsilon}$ gives more mass, more often we observe such values of $f$
around $x_{0}$. The notion becomes particularly natural from the view point of
interpretation in the stationary case, notion that we express here only in a
loose way: if the phenomenon at hand is independent of position $x_{0}$ and
maybe of a time parameter, $\nu_{x_{0},\epsilon}$ captures the oscillations of
$f$ in a statistical sense, where now ``statistical" may be interpreted in a
wider sense, not only locally spatial.

Rigorously, assume we have a sequence of measurable functions $\left(
f_{n}\right)  $ from $\mathbb{T}^{3}$ to $\mathbb{R}^{3}$ and a measurable
function $x\mapsto\nu_{x}$ from $\mathbb{T}^{3}$ to the space of probability
measures on Borel sets of $\mathbb{R}^{3}$ endowed with the topology of weak
convergence, such that
\[
\lim_{n\rightarrow\infty}\int_{\mathbb{T}^{3}}\varphi\left(  f_{n}\left(
x\right)  \right)  \psi\left(  x\right)  dx=\int_{\mathbb{T}^{3}}\left(
\int_{\mathbb{R}^{3}}\varphi\left(  y\right)  \nu_{x}\left(  dy\right)
\right)  \psi\left(  x\right)  dx
\]
for every $\varphi\in C_{b}\left(  \mathbb{R}^{3}\right)  $ and $\psi\in
C_{b}\left(  \mathbb{T}^{3}\right)  $. \textit{The measure-valued function
}$\left(  \nu_{x}\right)  $\textit{ is called the Young measure associated
with the sequence }$\left(  f_{n}\right)  $. The above convergence extends to
$\psi\in L^{2}\left(  \mathbb{T}^{3}\right)  $ by an easy argument of density,
taking into account that the sequence $\left\Vert \varphi\left(  f_{n}\left(
\cdot\right)  \right)  \right\Vert _{\infty}$ is bounded. Then we can take
$\psi=\frac{1}{\left\vert B\left(  x_{0},\epsilon\right)  \right\vert
}\one_{B\left(  x_{0},\epsilon\right)  }$ and get%
\[
\lim_{n\rightarrow\infty}\frac{1}{\left\vert B\left(  x_{0},\epsilon\right)
\right\vert }\int_{B\left(  x_{0},\epsilon\right)  }\varphi\left(
f_{n}\left(  x\right)  \right)  dx=\frac{1}{\left\vert B\left(  x_{0}%
,\epsilon\right)  \right\vert }\int_{B\left(  x_{0},\epsilon\right)  }\left(
\int_{\mathbb{R}^{3}}\varphi\left(  y\right)  \nu_{x}\left(  dy\right)
\right)  dx.
\]
In other words, if $\nu_{x_{0},\epsilon}^{n}$ is the local occupation measure
of $f_{n}$, we have
\[
\lim_{n\rightarrow\infty}\int_{\mathbb{R}^{3}}\varphi\left(  y\right)
\nu_{x_{0},\epsilon}^{n}\left(  dy\right)  =\frac{1}{\left\vert B\left(
x_{0},\epsilon\right)  \right\vert }\int_{B\left(  x_{0},\epsilon\right)
}\left(  \int_{\mathbb{R}^{3}}\varphi\left(  y\right)  \nu_{x}\left(
dy\right)  \right)  dx.
\]
More expressively, we may write%
\[
\lim_{n\rightarrow\infty}\nu_{x_{0},\epsilon}^{n}=\frac{1}{\left\vert B\left(
x_{0},\epsilon\right)  \right\vert }\int_{B\left(  x_{0},\epsilon\right)  }%
\nu_{x}dx
\]
in the weak sense of measures. This formula provides an interpretation for the
Young measure $\left(  \nu_{x}\right)  $ of the sequence $\left(
f_{n}\right)  $: \textit{its local average captures the oscillations of
}$f_{n}$\textit{ in a statistical sense, in the limit as }$n\rightarrow\infty
$. In the homogeneous cases, when $\nu_{x}$ is independent of $x$, the
interpretation is particularly strong.

\subsection{Large values}

In the case of a scalar transport equation, the values of solutions for
positive time cannot exceed the values of the initial conditions. 

In the case of a vector-advection equation, the stretching term may increase
the length of the vectors. The physical phenomenon of magnetic dynamo is an
example. The theory developed above may help to recognize the existence of
large values of the length of the vectors. Let us see a quantitative statement.

Assume we have%
\begin{align}\label{weak_convergence_measure_vector_fields}
    \lim_{n\rightarrow\infty}\int_{\mathbb{T}^{3}}\varphi\left(  B_t^{n}\left(
x\right)  \right)  \psi\left(  x\right)  dx=\int_{\mathbb{T}^{3}}\left(
\int_{\mathbb{R}^{3}}\varphi\left(  b\right)  \mu_{t,x}\left(  db\right)
\right)  \psi\left(  x\right)  dx
\end{align}
for a sequence of vector fields $\left(  B^{n}_t\right)  $, where $\varphi
,\psi,\mu_{t,x}$ are as above. Denote by $\mathcal{L}_{3}$ the Lebesgue measure
on $\mathbb{T}^{3}$. 

\begin{proposition} \label{PropLargeValues}
Assume that there exists a ball $B\left(  x_{0},r\right)  \subset
\mathbb{T}^{3}$ and real numbers $R,\lambda>0$ such that
\[
\mu_{t,x}\left(  B\left(  0,R\right)  ^{c}\right)  \geq\lambda\qquad\text{for
all }x\in B\left(  x_{0},r\right)  .
\] 
Then, for all $\lambda^{\prime}\in\left(  0,\lambda\right)  $ there exists
$n_{0}\in\mathbb{N}$ such that for all $n\geq n_{0}$
\[
\mathcal{L}_{3}\left\{  x\in B\left(  x_{0},r\right)  :\left\vert B_t^{n}\left(
x\right)  \right\vert \geq R\right\}  \geq\lambda^{\prime}\mathcal{L}%
_{3}\left(  B\left(  x_{0},r\right)  \right)  .
\]

\end{proposition}

\begin{proof}
For every $\delta\in\left(  0,\frac{R}{2}\right)  $, let $\varphi_{\delta}$ be
a smooth function on $\mathbb{R}^{3}$, equal to 1 on $B\left(  0,R-\delta
\right)  ^{c}$, zero on $B\left(  0,R-2\delta\right)  $, with values in
$\left[  0,1\right]  $ in $B\left(  0,R-\delta\right)  \backslash B\left(
0,R-2\delta\right)  $; and similarly let $\psi_{\delta}$ be a smooth function
on $\mathbb{T}^{d}$, equal to 1 on $B\left(  x_{0},r+\delta\right)  $, zero on
$B\left(  x_{0},r+2\delta\right)  $, with values in $\left[  0,1\right]  $ in
$B\left(  x_{0},r+2\delta\right)  \backslash B\left(  x_{0},r+\delta\right)
$. Then%
\[
\int_{\mathbb{R}^{3}}\varphi_{\delta}\left(  b\right)  \mu_{t,x}\left(
db\right)  \geq\lambda\qquad\text{for all }x\in B\left(  x_{0},r\right)
\]
and thus
\[
\int_{\mathbb{T}^{3}}\left(  \int_{\mathbb{R}^{3}}\varphi_{\delta}\left(
b\right)  \nu_{t,x}\left(  db\right)  \right)  \psi_{\delta}\left(  x\right)
dx\geq\lambda\mathcal{L}_{3}\left(  B_t\left(  x_{0},r\right)  \right)  .
\]
Therefore, given $\lambda^{\prime}\in\left(  0,\lambda\right)  $, there exists
$n_{0}\in\mathbb{N}$ such that for all $n\geq n_{0}$%
\[
\int_{\mathbb{T}^{3}}\varphi_{\delta}\left(  B_t^{n}\left(  x\right)  \right)
\psi_{\delta}\left(  x\right)  dx\geq\lambda^{\prime}\mathcal{L}_{3}\left(
B\left(  x_{0},r\right)  \right)  .
\]
We may easily deduce, by the arbitrariness of $\delta$, that
\[
\int_{B\left(  x_{0},r\right)  }\one_{\left\{  \left\vert B_t^{n}\left(  x\right)
\right\vert \geq R\right\}  }dx\geq\lambda^{\prime}\mathcal{L}_{3}\left(
B\left(  x_{0},r\right)  \right)  .
\]
This is the claimed formula. 
\end{proof}
Due to \autoref{main_thm_magnetic_field}, relation \eqref{weak_convergence_measure_vector_fields} is satisfied $\mathbb{P}-a.s.$ up to passing to subsequences. To be able to use \autoref{PropLargeValues} to say something about large values attained by the stochastic PDE for $B^n_t$, we prove here also that for every $x \in \T^3$ such that $\overline{B}_0(x)\neq 0$, the limit measure $\mu(t,x,db)$ has full support on $\R^3$ for every $t>0$. This fact does not follow trivially by a standard application of the Support Theorem because the diffusion matrix $\mathcal A(b)$ introduced in the proof of \autoref{prop:uniqueness_vlasov} is not differentiable at $b=0$. To remedy this fact we `avoid' the singularity by means of a localization procedure. We state the two main results we are going to use for the ease of the reader and refer to \cite[Theorem 3.5]{Millet1994}, \cite[Chapter 9.5]{baldi2017stochastic} for details. 
\begin{proposition}\label{support thm}
    Let $\sigma:\R^n\rightarrow\R^{n\times k}$ be of class $C^2$ with bounded first and second derivative. For $0\le t \le T$, let $X_t$ be the solution of 
    $$dX_t = \sigma(X_t)dW_t, \qquad X_0=x\in \R^n,$$ then for every $\alpha \in [0, 1/2)$, the support of the measure $P\circ X^{-1}$ in $C^\alpha(0, T; \R^n)$ is the closure of the set $\left\{S(h), h\in \H \right\}$ where $\H$ is the Cameron-Martin space of $W$ and $S(h)$ is given by the solution of 
   \begin{equation}\label{control eq}
       S(h)_t = x + \int_0^t \sigma(S(h)_t)\cdot \dot h_tdt - \frac{1}{2}\int_0^t \nabla(\sigma(S(h)_t))\sigma(S(h)_t)dt
   \end{equation}
\end{proposition}
\begin{proposition}\label{localization}
    Let $\sigma_i(x):\R^n\rightarrow\R^{n\times k}, \ i=1,2 $ be measurable functions and let $X^i$ be the solutions of
    $$dX^i_t = \sigma_i(X^i_t)dW_t \qquad X^i_0=x_0.$$
    Let $D\subset \R^d$ be an open set such that on $D$ it holds $\sigma_1 = \sigma_2$ and 
    $$|\sigma_i(x)- \sigma_i(y)|\le L|x-y|.$$
    Then, if $\tau_i$ denotes the exit time of $X_i$ from $D$, we have 
    \begin{equation}
        \tau_1=\tau_2 \text{ a.s. and }\ \PP(X_i(t)=X_2(t), \text{ for every }0\le t\le \tau_1) = 1.
    \end{equation}
\end{proposition}
Recalling the notation introduced in the proof of \autoref{prop:uniqueness_vlasov}, with these tools we prove the following
\begin{proposition}\label{support lemma}
    Let $b_t$ be the solution of 
    \begin{equation}
        db_t = \A(b_t)d\hat{W}_t
    \end{equation}\label{sde}
    with $b(0)=b_0\neq 0$. Then for every $T>0$, $\eps >0 $, $b^*\in \R^3$ we have 
    $$\PP(|b_T-b^*| \le \eps)>0$$
\end{proposition}
\begin{proof}
Without loss of generality we can suppose that $b^*\neq 0$ (if it does, choose $|b'|\le \frac{\eps}{2}$, then use $b'$ as target point). Fix $\delta \le \frac{1}{2}(|b_0|\land |b^*|)$ , $R\ge 2(|b_0|\lor |b^*|) $ and call $D= B(0, R)\setminus \overline{B(0, \delta)} $ the open annulus of radii $\delta$ and $R$. Set
$\tau = \inf \{t>0 : b_t \notin D\} $. Now consider a new diffusion coefficient $\mathcal{A}^r$ such that $\mathcal{A}^r(b)=\mathcal{A}(b)$ for every $b\in D$ and  $\A^r \in C^\infty_c(\R^3)$.  Finally consider $b_t^r$ the solution of $db_t^r = \A^r(b_t)d\hat{W}_t$ with $b_0^r= b_0$, and $\tau^r$ its associated exit time from $D$. By \autoref{localization}, we know that $\tau = \tau^r$ a.s. and $b_t=b_t^r$ a.s. for every $t\le \tau$. This implies that
\begin{align*}
    \PP(|b_T-b^*| \le \eps) &\ge \PP(|b_T-b^*| \le \eps, \ T\le \tau, \ \sup_{t\le T}|b_t-b_t^r|=0)\\
    &\ge \PP(|b^r_T-b^*| \le \eps, \  T\le \tau^r ).
\end{align*}
Let us now suppose that the line $y_t= \frac{t}{T}b^* + \frac{T-t}{T}b_0$ lies inside $D^\eps:=\{x \in D : |x|>\delta + \eps\}$. In that case we have 
\begin{align*}
 \PP(|b^r_T-b^*| \le \eps,\ T\le \tau^r ) &\ge  \PP(|b^r_T-b^*| \le \eps,\ T\le \tau^r,\ \sup_{t\le T}|b_t^r - y_t|\le \eps/2 )  \\
    &= \PP( \sup_{t\le T}|b_t^r - y_t|\le \eps/2 ,\ T\le \tau^r)\\
    &= \PP( \sup_{t\le T}|b_t^r - y_t|\le \eps/2).
\end{align*}
Thus, we are left to show that the curve $y_t$ is in the support of the law of $b^r$ on $C^0(0, T; \R^3)$. But now we are in shape to apply \autoref{support thm}: we just need to check that there exists $h \in \H$ such that $y_t$ solves the control problem \ref{control eq}. This is readily done, setting 
\begin{equation}
    h_t = \int_0^t \A^{-1}(y_t)\pa{\frac{b^*- b_0}{T} - (\nabla \A(y_t))\A(y_t)}dt
\end{equation}
Since $y_t$ has support in $D$, $\A(y_t)$ is invertible with bounded inverse for every $t\le T$. Also $\A$ and $\nabla \A$ are bounded, thus the integrand in is $L^2$ and $h \in \H$. \\
If instead $y_t$ does not lie inside of $D$, we can find a middle point $b^{**}$ such that the line $\tilde y_t$ connecting $b_0, b^{**}$ and $b^*$ in time $T$ lies inside of $D^\eps$ and repeat the argument using $\tilde y_t$ in place of $y_t$. 
\end{proof}
To be able to apply \autoref{PropLargeValues} we need some uniformity of the previous statement with respect to the initial point $x\in \T^3$.
\begin{lemma}\label{continuity_measure_data}
     Suppose that $B_0\in C(\T^3; \R^3)$ then for every $R>0$ the map $\T^3 \ni x\rightarrow \mu_{t,x}(B(0, R))$ is continuous. 
\end{lemma}
\begin{proof}
    Denote $b_t^\nu$ the same process as above with initial condition $\nu \in \R^3$. Since the matrix field $\A$ is Lipschitz continuous and $B_0$ is continuous by assumption, $b_t^{B_0(x_n)} \rightarrow b_t^{B_0(x)} $ a.s. if $x_n \rightarrow x$.  The result follows then by dominated convergence recalling that $$\mu_{t,x}(B(0,R))= \EE[\one_{B(0, R)}(b_t^{B_0(x)})].$$ 
\end{proof}
\begin{corollary} \label{Corollary B6 on support}
    If $\overline{B}_0\in V\cap C(\T^3;\R^3)$ and is not identically null, the family of measures $\mu_{t,x}$ satisfy the assumption of \autoref{PropLargeValues} for any value of $R>0$. 
\end{corollary}
\begin{proof}
    By \autoref{support lemma}, for every $x_0\in \T^3$ such that $\overline{B}_0(x_0)\neq 0$ and $R>0$ there exist $\lambda >0$ such that $\mu_{t,x_0}(B(0, R)^c)> \lambda$, then by continuity of this quantity with respect to $x$ ensured by the \autoref{continuity_measure_data}, there exists $r>0$ such that $\mu_{t,x}(B(0, R)^c)\ge \lambda/2$ for every $x \in B(x_0, r)$. 
\end{proof}
\begin{acknowledgements}
E.L. wishes to thank Luca Gennaioli for the fruitful discussion about the weak topology on the space of probability measures.  The research of F.F. and Y.T. is
funded by the European Union (ERC, NoisyFluid, No. 101053472). Views and
opinions expressed are however those of the authors only and do not necessarily
reflect those of the European Union or the European Research Council. Neither
the European Union nor the granting authority can be held responsible for them.
\end{acknowledgements}
\bibliography{biblio}{}
\bibliographystyle{plain}

\end{document}